\documentclass[12pt,a4paper]{amsart}
\usepackage[utf8]{inputenc}
\usepackage[T1]{fontenc}
\usepackage[english]{babel}
\usepackage[all]{xy}
\usepackage{fancyhdr}
\usepackage{array}
\usepackage{amsfonts}
\usepackage{bbm}
\usepackage{amsmath}
\usepackage{amssymb}
\usepackage{graphicx}
\usepackage{stmaryrd}
\usepackage{moreverb}
\usepackage{geometry}
\usepackage{amsthm}
\usepackage{color}
\usepackage{bbold}
\usepackage[style = alphabetic, maxbibnames = 99, backend = bibtex]{biblatex}
\usepackage{yhmath}
\usepackage{tikz}
\usepackage{hyperref}
\usepackage{cleveref}
\usepackage{csquotes}

\hypersetup{breaklinks = true}

\title{Miyaoka--Yau equality and uniformization of log Fano pairs}
\author{Louis Dailly}
\address{Louis Dailly, Institut de Math\'ematiques de Toulouse, Universit\'e Paul Sabatier, 31062 Toulouse Cedex~9, France}
\email{\href{mailto:louis.dailly@univ-rennes.fr}{louis.dailly@univ-rennes.fr}}

\newcommand{\incl}[1][r]
  {\ar@<-0.2pc>@{^(-}[#1] \ar@<+0.2pc>@{-}[#1]}
\newcommand{\immouv}[1][r]
   {\ar@{}[#1] |*[o][F]{\hbox{%
         \vrule width 1.5mm height 0pt depth 0pt%
         \vrule width 0pt height .75mm depth .75mm%
         }}
     \ar@{^{(}->}[#1]}

\newcommand{\N}{\mathbb{N}}
\newcommand{\Z}{\mathbb{Z}}
\newcommand{\Q}{\mathbb{Q}}
\newcommand{\R}{\mathbb{R}}
\newcommand{\C}{\mathbb{C}}
\newcommand{\D}{\mathbb{D}}
\newcommand{\PP}{\mathbb{P}}

\newcommand{\reg}{\mathrm{reg}\,}
\newcommand{\orb}{\mathrm{orb}\,}

\newcommand{\End}{\mathrm{End}\,}

\newcommand{\PGL}{\mathrm{PGL}\,}

\newcommand{\inv}{\mathrm{inv}}
\newcommand{\ch}{\mathrm{ch}}

\theoremstyle{definition}
\newtheorem{theo}{Theorem}[section]
\newtheorem{prop}[theo]{Proposition}
\newtheorem{coro}[theo]{Corollary}
\newtheorem{defi}[theo]{Definition}
\newtheorem{lemm}[theo]{Lemma}

\newtheorem{rema}[theo]{Remark}
\newtheorem{exem}[theo]{Example}
\newtheorem*{ackn}{Acknowledgement}

\theoremstyle{plain}
\newtheorem*{theox}{Main Theorem}

\newenvironment{itemize*}
    {\begin{itemize}%
      \setlength{\itemsep}{0pt}%
      \setlength{\parskip}{0pt}}%
    {\end{itemize}}

\bibliography{mabiblio}
\AtBeginBibliography{\tiny}

\begin{document}

\begin{abstract}
Let $(X, \Delta)$ be a log Fano pair with standard coefficients. We show that if it satisfies the equality case in Miyaoka--Yau inequality, then its orbifold universal cover is a projective space.
\end{abstract}

\maketitle

\tableofcontents

\section{Introduction}

At the beginning of the 20th century, Poincar\'e and Koebe stated that the projective line $\PP^1(\C)$ is the unique Riemann surface which is compact connected, simply connected. In higher dimensions, several characterizations of complex projective spaces exist. We can think about Mori's characterization \cite{Mor79}: projective spaces are the only compact complex manifolds which have an ample tangent bundle, or Siu--Yau's one \cite{SY18}: projective spaces are the only manifolds which can be endowed with a Kähler metric whose bisectional holomorphic curvature is positive.

In another direction, we can give a characterization of projective spaces coming from uniformization theory.
Let us describe the universal cover of a compact complex manifold $X$ of dimension $n$, if it satisfies the following two geometric assumptions: 
\begin{itemize}
    \item $(KE_\lambda)$ : the manifold $X$ can be endowed with a Kähler--Einstein metric $\omega$: $\mathrm{Ric}\, \omega = \lambda \omega$, with $\lambda \in \R$.
    \item $(MY)$ : the manifold $X$ satisfies the so-called Miyaoka--Yau equality: 
    $$
    \left( 2(n+1) c_2(X) - n c_1(X)^2 \right) \cdot \left[ \omega \right]^{n-2} = 0.
    $$
\end{itemize}
Under these assumptions, the universal cover $\widetilde{X}$ of $X$ is biholomorphic to:
\begin{itemize}
    \item the complex projective space $\PP^n$ if $\lambda > 0$, 
    \item the complex affine space $\C^n$ if $\lambda = 0$,
    \item the unit ball $\mathbb{B}^n = \left\lbrace z \in \C^n \,\,|\,\, |z|<1 \right\rbrace$ if $\lambda < 0$.
\end{itemize}

In the case $\lambda > 0$, we get that the manifold $X$ is biholomorphic to $\PP^n$, as $X$ is a Fano manifold which is always simply connected. Hence, it yields the desired characterization of $\PP^n$.
Now, if we consider a non-trivial finite group $G$, it can be realized as a subgroup of $\PGL(n+1, \C)$ for some integer $n$, and we can consider the orbit space $\PP^n / G$. The group $G$ does not act freely, and the quotient map $\PP^n \longrightarrow \PP^n / G$ is ramified. The space $\PP^n / G$ is a normal projective variety with quotient singularities endowed with a branching divisor $\Delta_G$. We get an orbifold pair $(\PP^n / G, \Delta_G)$ that is uniformized by $\PP^n$.
A natural question is: can we adapt the previous characterization of quotients to the pair framework ? More precisely, if a pair $(X, \Delta)$ satisfies singular versions of $(KE_\lambda)$ and $(MY)$, can we describe its universal cover $\left( \widetilde{X}_\Delta, \widetilde{\Delta} \right)$ ?

In the case $\lambda \leq 0$, it has been shown \cite[Thms.~A and B]{CGG24} that such characterizations hold. In the case $\lambda > 0$, the uniformization result holds when $\Delta = 0$ as a consequence of \cite[Thm.~1.5]{GKP22} and \cite[Thm.~B]{DGP24}.
The aim of the present paper is to investigate the case $\lambda > 0$ with a boundary divisor $\Delta$.

On the one hand, the authors of \cite{GKP22} show the uniformization result by replacing the assumption $(KE_{>0})$ with the condition of semistability of the canonical extension, which is an extension of the tangent sheaf $\mathcal{T}_X$ by $\mathcal{O}_X$. On the other hand, the authors of \cite{DGP24} show that the existence of a singular K\"ahler--Einstein metric implies the above mentioned semistability.

The aim of our article is to show the following theorem, which is analogous to \cite[Thm.~1.3]{GKP22} with a boundary divisor $\Delta$.

\begin{theox}
    Let $(X, \Delta)$ be a log Fano pair of dimension $n$.
    There exists a finite group $G \subseteq \PGL(n+1, \C)$, such that $(X, \Delta) \simeq (\PP^n / G, \Delta_G)$ if and only if the following two conditions hold:\\
    $(i)$ The canonical extension $\mathcal{E}_{(X, \Delta)}$ is semistable with respect to $-(K_X + \Delta)$,\\
    $(ii)$ The pair $(X, \Delta)$ satisfies the Miyaoka--Yau equality:
    $$
    (2(n+1) c_2(X, \Delta) - n c_1(X, \Delta)) \cdot c_1(X, \Delta)^{n-2} = 0.
    $$
\end{theox}

\emph{Remarks on the theorem:}

It is a priori unclear what object $\mathcal{E}_{(X, \Delta)}$ is and what semistability means in assumption $(i)$. In the following, the assumption $(i)$ is: \emph{There exists a strictly $\Delta$-adapted morphism $f_0 \colon Y_0 \longrightarrow X$ such that $\mathcal{E}_{X, \Delta, f_0}$ is semistable with respect to $f_0^* (-(K_X + \Delta))$.}

This assumption is related to the notion of \emph{adapted canonical extension} $\mathcal{E}_{X, \Delta, f}$ for a strictly $\Delta$-adapted morphism $f \colon Y \longrightarrow X$. This notion is introduced in the third section. It is a reflexive sheaf on $Y$ which is an extension of the adapted tangent sheaf $\mathcal{T}_{X, \Delta, f}$ by $\mathcal{O}_Y$. This is a notion similar to the adapted differentials $\Omega^{[1]}_{X, \Delta, f}$ introduced by Miyaoka and Campana (see \cite[Sect.~I.3]{CKT16}).

The second assumption $(ii)$ is a singular version of $(MY)$ that uses the definition of orbifold Chern classes initiated by Mumford. Here, we use an equivalent definition using orbifold de Rham cohomology, see \cite[Sect.~3]{CGG24}.

The question of uniformization in \cite{GKP22}, \cite{DGP24}, \cite{CGG24} has been considered in the more general setting of klt pairs. A part of the proof of our theorem is devoted to show that a pair satisfying assumptions $(i)$ and $(ii)$ has only quotient singularities.\\

\emph{Strategy of the proof:}

The strategy of the proof is based on \cite[Sect.~5.2]{GKP22}. In this article, the authors introduce a notion of canonical extension $\mathcal{E}_X$ for a klt space $X$. Then, they show that the Miyaoka--Yau equality implies projective flatness of $\mathcal{E}_X$ on the regular locus of the variety, hence flatness of $\End(\mathcal{E}_X)$. Hence, they introduce the "maximally quasi-\'etale" cover $\pi \colon Y \longrightarrow X$ \cite[Thm.~1.14]{GKP16} that plays the role of a global uniformization of $X$. The goal is to show that $Y = \PP^n$. As previously, we have the flatness of the endomorphism bundle associated $\End(\mathcal{E}_Y)$ over $Y_\reg$. Because $Y$ is maximally quasi-\'etale, $\End(\mathcal{E}_Y)$ extends to a locally free sheaf on $Y$. As a consequence of the Zariski--Lipman theorem for klt spaces, $Y$ is smooth as $\mathcal{T}_Y$ is locally a direct summand of $\End(\mathcal{E}_Y)$. It is simply connected because it is a Fano manifold. Hence, there exists a line bundle $L$ on $Y$ such that $\mathcal{E}_Y = L^{\oplus n+1}$, $\mathcal{E}_Y$ being projectively flat on a simply connected manifold. Then, we have $c_1(Y) = c_1(\mathcal{E}_Y) = (n+1) c_1(L)$ and we conclude by Kobayashi--Ochiai's characterization of $\PP^n$.

In the pair setting, two main difficulties arise from the presence of a boundary divisor $\Delta$:
\begin{itemize}
    \item The universal cover $(\widetilde{X}_\Delta, \widetilde{\Delta})$ of an orbifold could still have singularities in codimension 1 given by the boundary divisor $\widetilde{\Delta}$.
\end{itemize}

Under the assumptions of the theorem, a Zariski--Lipman type argument is needed to show that the pair is smooth i.e $\widetilde{\Delta} = 0$ and $\widetilde{X}_\Delta$ is smooth. In the proof of \cite{CGG24} for the case $\lambda < 0$, this is obtained by arguments coming from uniformization of Riemannian orbifolds, using the constancy of the curvature. In the case $\lambda = 0$ in the same article, the authors use the existence of an orbi-\'etale map via the existence of a section of some large power of $K_X + \Delta$. Recently, the authors of \cite{GP24} study system of Hodge-bundles in the pair setting to study quotient of bounded symmetric domains.

In our article, we state in the first section a criterion related to the notion of orbi-locally freeness of the tangent bundle. Then, we construct the canonical extension $\mathcal{E}_{(X, \Delta)}$ of an orbifold $(X, \Delta)$. It is an extension of $\mathcal{T}_{(X, \Delta)}$ by $\mathcal{O}_X$. From this criterion, it remains to show that the orbibundle $\End(\mathcal{E}_{(X, \Delta)})$ is flat in order to prove that $(X, \Delta)$ is developable.
\begin{itemize}
    \item In order to describe the geometry of the pair $(X, \Delta)$, we need to deal with orbisheaves and orbibundles that are virtual objects on $X$, as they make sense in local uniformizations.
\end{itemize}

A way to deal with concrete sheaves is to pull them back on strictly adapted coverings $f \colon Y \longrightarrow (X, \Delta)$. We give a general construction of the adapted sheaf $\mathcal{F}_{X, \Delta, f}$ associated to an orbibundle $\mathcal{F}$ on $(X, \Delta)$, and compute orbifold Chern classes for such sheaves.

In order to show that a klt pair $(X, \Delta)$ is orbifold under assumptions $(i)$ and $(ii)$ of the theorem, we construct a smooth orbi-\'etale orbistructure following the strategy of \cite{CGG24}. We introduce the above mentioned adapted canonical extension $\mathcal{E}_{X, \Delta, f}$ to show that the local covers $Y$ are smooth away from a divisor. This canonical extension is a sheaf on $Y$, obtained as an extension of the adapted tangent bundle $\mathcal{T}_{X, \Delta, f}$ by $\mathcal{O}_Y$, that takes into account the boundary divisor and additional ramifications. 

Moreover, we show that some classical adapted sheaves constructed in the literature (the cotangent sheaf $\Omega^{[1]}_{X, \Delta, f}$, the tangent sheaf $\mathcal{T}_{X, \Delta, f}$ or the canonical extension $\mathcal{E}_{X, \Delta, f}$) are adapted sheaves associated to orbibundles on $(X, \Delta)$.\\

\emph{Organisation of the paper:}

In the first part, we recall some useful facts on orbifolds, and give a construction of the canonical extension $\mathcal{E}_{(X, \Delta)}$ of an orbifold $(X, \Delta)$. Then, we state the above mentioned criterion of developability for a pair $(X, \Delta)$.

In a second part, we deal with the klt pair setting. We first introduce the adapted canonical extension $\mathcal{E}_{X, \Delta, f}$ where $f \colon Y \longrightarrow X$ is a strictly $\Delta$-adapted morphism, and we construct a canonical extension $\mathcal{E}_{(X, \Delta)}$ for a klt pair. Then, we give another interpretation of this sheaf in terms of pullback of the orbisheaf $\mathcal{E}_{(X, \Delta)}$ above the orbifold locus, that provides us with a way to compute orbifold Chern classes of this sheaf.

The third part is devoted to the proof of the theorem.

The last part introduce some examples of pairs $(\PP^n / G, \Delta_G)$, like weighted projective spaces, and we classify log smooth quotients of $\PP^2$ by computing explicitly Miyaoka--Yau discriminants.

\begin{ackn}
This work is a part of my PhD thesis. I warmly thank my advisors Benoît Claudon and Henri Guenancia for their support through discussions, suggestions during my research and corrections of the present paper.
\end{ackn}

\section{Orbifolds}

\subsection{Generalities on pairs}

In this section, we recall some useful definitions following the exposition of \cite{CGG24}.

\subsubsection{Pairs and orbistructures}

In order to keep notation light, we use same notation for a variety $X$ and its underlying complex space. The fundamental group of such a variety $\pi_1(X)$ will be the topological fundamental group of the associated analytic space.

\begin{defi}
    A pair $(X, \Delta)$ is the datum of a normal projective variety (or a normal analytic space) $X$ and a Weil $\Q$-divisor $\Delta$ with standard coefficients. It means that it is of the form $\displaystyle  \Delta := \sum_i \left( 1 - \frac{1}{m_i} \right) \Delta_i$ where $\Delta_i$ are irreducible divisors such that the union is locally finite, and we have $m_i \in \N_{\geq 2}$.

    We will denote by $X^* := X_\reg \backslash |\Delta|$ \emph{the regular locus of $(X, \Delta)$}.
\end{defi}

\begin{defi}
    A morphism $f \colon Y \longrightarrow X$ between complex normal spaces is called \emph{quasi-\'etale} (resp.~\emph{quasi-flat}) if there exists an analytic subset $Z \subseteq X$ of codimension at least 2 such that $f$ is \'etale (resp. flat) over $X \backslash Z$.
\end{defi}

Quasi-\'etaleness and quasi-flatness are invariant under finite base change. More generally, a finite map between normal varieties $f \colon Y \longrightarrow X$ is quasi-flat, by studying normal forms of $f$ at general points of the components of codimension 1 of the branch locus.

\begin{defi}
    We say that a finite morphism $f \colon Y \longrightarrow X$ between complex normal spaces is \emph{Galois} if there exists a finite group $G \subseteq \textrm{Aut}(Y)$ such that $f$ is isomorphic to the quotient map $\pi_G \colon Y \longrightarrow Y/G$.
\end{defi}

\begin{defi}
    Let $(X, \Delta)$ be a pair and let $Y$ be a normal analytic space. We say that a Galois morphism $f \colon Y \longrightarrow X$ is:
    \begin{itemize}
        \item[•] \textit{$\Delta$-adapted} if for all $i$, there exists a reduced divisor $\Delta'_i$ on $Y$ and $a_i \in \N_{\geq 1}$ such that $f^* \Delta_i = m_i a_i \Delta'_i$ 
        \item[•] \textit{strictly $\Delta$-adapted} if $f$ is $\Delta$-adapted and for all $i$, $a_i = 1$.
        \item[•] \textit{orbi-\'etale} if $f$ is strictly $\Delta$-adapted and \'etale over $X^*$.
    \end{itemize}
\end{defi}

\noindent
A morphism $f \colon Y \longrightarrow X$ is orbi-\'etale for the pair $(X, 0)$ if and only if it is quasi-\'etale.

\begin{defi}\cite[Def.~9]{CGG24}\label{deforbistructure}
    Let $(X, \Delta)$ be a pair. An \textit{orbistructure on $(X, \Delta)$} is a collection $\mathcal{C}$ of triples $(U_\alpha, V_\alpha, \eta_\alpha)$ called local uniformizing charts (l.u.c.) such that:
    \begin{itemize}
        \item[•] $(U_\alpha)_\alpha$ is an open cover of $X$,
        \item[•] $V_\alpha$ is a normal complex space,
        \item[•] $\eta_\alpha \colon V_\alpha \longrightarrow U_\alpha$ is an adapted morphism for $(U_\alpha, \Delta_{| U_\alpha})$, satisfying the following compatibility condition. If we denote by $V_{\alpha\beta}$ the normalization of $V_\alpha \underset{X}{\times} V_\beta$, then the natural morphisms $g_{\alpha\beta} \colon V_{\alpha\beta} \longrightarrow V_\alpha$ are quasi-\'etale.
    \end{itemize}
    We say that an orbistructure is \textit{strict} (resp. \textit{orbi-\'etale}) if each $\eta_\alpha$ is strictly-adapted (resp. orbi-\'etale). The orbistructure is said \emph{smooth} if each $V_\alpha$ is smooth. If $(X, \Delta)$ is endowed with a smooth orbi-\'etale orbistructure, we say that $(X, \Delta)$ is an \emph{orbifold}.
\end{defi}

\begin{rema}
    When $(X, \Delta)$ is endowed with a smooth orbi-\'etale orbistructure, the branch locus of the maps $g_{\alpha\beta}$ is of pure codimension 1 by Nagata's purity of the branch locus. Since the maps $g_{\alpha\beta}$ are quasi-\'etale, they are actually \'etale.
    More generally, when $\alpha = \beta$, we can observe that there exists a surjective map $(x, g) \in V_\alpha \times G_\alpha \longmapsto (x, g \cdot x) \in V_{\alpha} \underset{X}{\times} V_{\alpha}$ that produces a morphism $V_\alpha \times G_\alpha \longrightarrow V_{\alpha\alpha}$ by universal properties of normalization. But this morphism is finite of degree one above the \'etale locus of $\eta_\alpha$. By Zariski's main theorem, it is an isomorphism so that $V_{\alpha\alpha} = V_\alpha \times G_\alpha$, and the projections are \'etale.
\end{rema}

\begin{defi}
    Let $(X, \Delta)$ be a pair. We denote by $(X, \Delta)_\orb$ \emph{the orbifold locus of $(X, \Delta)$}. It is the biggest open subset $U \subseteq X$ such that $(U, \Delta_{|U})$ can be endowed with a smooth orbi-\'etale orbistructure.
\end{defi}

\subsubsection{Orbifold morphism and pullback orbistructure}

We have a notion of orbifold morphism between pairs endowed with orbistructures.

\begin{defi}\label{orbifoldmorphism}
    Let $(Y, \Delta'), (X, \Delta)$ be two pairs, endowed with orbistructures $\mathcal{D}$ and $\mathcal{C}$ respectively. We say that a morphism $\pi \colon Y \longrightarrow X$ is \textit{orbifold at $y \in Y$} if for each l.u.c. $(U,V, \eta) \in \mathcal{C}$ around $\pi(y)$, there exists a l.u.c. $(T, W, \mu) \in \mathcal{D}$ around $y$ and a morphism $\pi_{\eta \mu} \colon W \longrightarrow V$ such that $\eta \circ \pi_{\eta \mu} = \pi \circ \mu$.
\end{defi}

When we have an orbifold morphism $\pi \colon (Y, \Delta') \longrightarrow (X, \Delta)$, and two l.u.c. denoted by $(U_\alpha, V_\alpha, \eta_\alpha) \in \mathcal{C}$ and $(T_{\alpha'}, W_{\alpha'}, \mu_{\alpha'}) \in \mathcal{D}$, we will denote $\pi_{\alpha\alpha'} := \pi_{\eta_\alpha \mu_{\alpha'}}$ in order to simplify notation.

\begin{rema}
    Here we follow the definition of maps from \cite{Sat56}. But there are several definitions of orbifold morphisms (see \cite{MP97} for the definition of strong morphisms).
\end{rema}

\begin{rema}\label{remarkoforbifoldmorphism}
     If we have such an orbifold morphism $\pi \colon (Y,\Delta') \longrightarrow (X, \Delta)$, we get a morphism $\pi_{\alpha\beta, \alpha'\beta'} \colon W_{\alpha'\beta'} \longrightarrow V_{\alpha\beta}$ by universal properties of fiber product and normalization:
     $$
    \xymatrix {
    W_{\alpha\beta} \ar@{.>}[rrrr]^{\exists \, \pi_{\alpha\beta, \alpha'\beta'}} \ar[dd]_{l_{\alpha'\beta'}} \ar[dr]^{l_{\beta'\alpha'}} &&&& V_{\alpha\beta} \ar[dr] \ar[dd] |!{[dl];[dr]}\hole & \\
    & W_{\beta'} \ar[rrrr]_{\pi_{\beta\beta'}} \ar[dd]^--(.35){\mu_{\beta'}} &&&& V_\beta \ar[dd] \\
    W_{\alpha'} \ar[rrrr]_{\pi_{\alpha\alpha'}} |!{[ur];[dr]}\hole \ar[dr]_{\mu_{\alpha'}} &&&& V_\alpha \ar[rd] & \\
    & Y \ar[rrrr]_\pi &&&& X \\
  }
    $$    
\end{rema}

There is a way to pull back orbistructures by $\Delta$-adapted morphisms.

\begin{defi}\label{pullbackstructure}
    Let $(X, \Delta)$ be a pair with an orbi-\'etale orbistructure $\mathcal{C}$, and $\pi \colon Y \longrightarrow X$ be a $\Delta$-adapted morphism. The \emph{pull-back orbistructure}, denoted by $\mathcal{C}_\pi$ is an orbistructure on $(Y, 0)$ given by the triples $(\pi^{-1}(U_\alpha), W_\alpha, \mu_\alpha )$, where $W_\alpha$ is the normalization of $Y \underset{X}{\times} V_\alpha$, and $\mu_\alpha$ is the morphism $W_\alpha \longrightarrow Y$. 
\end{defi}

One can show that the collection defining $\mathcal{C}_\pi$ is an orbistructure because it satisfies the compatibility condition on intersections. Indeed, let $l_{\alpha\beta} \colon W_{\alpha\beta} \longrightarrow W_\alpha$ be the maps from the normalization of $W_\alpha \underset{Y}{\times} W_\beta $ to $W_\alpha$. In normal forms above general points of $\Delta$ and additional branching locus, we see that the morphisms $l_{\alpha\beta}$ are \'etale. It follows that $l_{\alpha\beta}$ are quasi-\'etale.

Moreover, the morphism $\pi$ is naturally an orbifold morphism with respect to orbistructures $\mathcal{C}_\pi$ and $\mathcal{C}$.

\subsubsection{Fundamental group and universal cover in the orbifold setting}

\begin{defi}
    Let $(X, \Delta)$ be a pair. Locally around a general point of an irreducible component $\Delta_i$, the model of $X^*$ is $\D^* \times \D^{n-1}$ (here $\D := \left\lbrace z \in \C \,\,|\,\, |z| < 1 \right\rbrace$) and its fundamental group is $\Z$. Let us denote by $\gamma_i$ a loop that generates this group and by $\langle\langle \gamma_i^{m_i} \rangle\rangle$ the normal subgroup of $\pi_1(X^*)$ generated by the homotopy classes of the loops $\gamma_i^{m_i}$. The \textit{orbifold fundamental group} $\pi_1^{\orb}(X, \Delta)$ is the group defined by:
    $$
    \pi_1^{\orb}(X, \Delta) := \pi_1(X^*) / \langle\langle \gamma_i^{m_i} \rangle\rangle.
    $$
\end{defi}

We refer to \cite[Section~2.3]{CGG24} for definitions of coverings branched along $\Delta$ and the geometric Galois correspondence between covers branched along $\Delta$ and subgroups of $\pi_1^\orb(X, \Delta)$. From this correspondence, the definition of universal cover follows.

\begin{defi}
    Let $(X, \Delta)$ be an orbifold. \textit{The universal cover of $(X, \Delta)$} is the orbifold $( \widetilde{X}_\Delta, \widetilde{\Delta})$ corresponding to the subgroup $\lbrace 1 \rbrace \subseteq \pi_1^{\orb}(X, \Delta)$ under the correspondence of Theorem 2.16 of the above mentioned article.
\end{defi}

\begin{rema}\label{orbistructureofuniversalcover}
    The smooth orbi-\'etale orbistructure on $(\widetilde{X}_{\Delta}, \widetilde{\Delta})$ is given by a collection $(\widetilde{U}_\alpha, V_\alpha, \widetilde{\eta}_\alpha)$, where $V_\alpha$ are the same spaces that define the orbistructure on $(X, \Delta)$, and morphisms $\widetilde{\eta}_\alpha$ are quotient maps associated to a subgroup $H_\alpha$ of $G_\alpha$.
\end{rema}

\subsection{Canonical extension of an orbifold \texorpdfstring{$(X, \Delta)$}{(X,D)}}

In this section, we construct the canonical extension of an orbifold $(X, \Delta)$.

\subsubsection{Generalities on orbisheaves}

In this section, we introduce a notion of orbisheaves, which is slightly different from the definition of \cite[Def.~2.18]{GT22}, as we ask for an isomorphism between the reflexive pullbacks (it means the reflexive hulls of pullbacks) on intersections. Our definition is well-suited to our context because    we consider only finite morphisms that preserve reflexive sheaves on big open subsets. This is a consequence of the quasi-flatness of finite morphisms, and the fact that pulling back sheaves by flat morphisms preserve reflexive sheaves. Hence, we can identify reflexive hulls on the whole space. This procedure has been used to identify adapted sheaves with some reflexive pullbacks (e.g., see \cite[Lemma.~18]{CGG24}). Moreover, that is why the definition of pullbacks of orbisheaves we introduce requires finiteness of morphisms.

Our version of orbisheaves is weaker than the one introduced by Guenancia and Taji, but it is not well-behaved with more singular orbisheaves (namely, taking the reflexive hull annihilates torsion phenomena). Moreover, not taking reflexive hulls defines pullbacks of orbisheaves in a more natural way without requiring finiteness.

Both versions coincide when the transition morphisms $g_{\alpha\beta}$ are flat, for example when $(X, \Delta)$ is an orbifold.

\begin{defi}
    Let $(X, \Delta)$ be a pair endowed with an orbistructure $\mathcal{C}$, whose l.u.c. are denoted $(U_\alpha, V_\alpha, \eta_\alpha)_\alpha$. We use notation of Definition~\ref{deforbistructure}. \emph{An orbisheaf $\mathcal{F}$} is a collection $(\mathcal{F}_\alpha)_\alpha$ of sheaves such that:
    \begin{itemize}
    \item for each $\alpha$, $\mathcal{F}_\alpha$ is a coherent $\mathcal{O}_{V_\alpha}$-module on $V_\alpha$,
    \item for each $\alpha, \beta$ there is an isomorphism, called \emph{transition function}, of $\mathcal{O}_{V_{\alpha\beta}}$-modules $\varphi_{\alpha\beta} : g_{\beta\alpha}^{[*]} \mathcal{F}_\beta \longrightarrow g_{\alpha\beta}^{[*]} \mathcal{F}_\alpha $ satisfying compatibility conditions on triple intersections (here, the notation $g_{\alpha\beta}^{[*]} \mathcal{F}_\alpha$ is used to denote the reflexive hull of the pullback $\left( g_{\alpha\beta}^{*} \mathcal{F}_\alpha \right)^{\vee\vee}$).
    \end{itemize}
    We say that $\mathcal{F}$ is an \emph{orbibundle of rank $r$} (resp. \emph{torsion free, reflexive}) if each $\mathcal{F}_\alpha$ is locally free of rank $r$ (resp. torsion free, reflexive).
    If $\mathcal{F}$ and $\mathcal{G}$ are two orbisheaves, a \emph{morphism of orbisheaves} $u : \mathcal{F} \longrightarrow \mathcal{G}$ is a collection of morphisms $(u_\alpha)_\alpha$ of morphisms of $\mathcal{O}_{V_\alpha}$-modules $u_\alpha : \mathcal{F}_\alpha \longrightarrow \mathcal{G}_\alpha$ compatible with transition morphisms. See \cite[Def.~2.21]{GT22} for details.

    A morphism of orbisheaves $u : \mathcal{F} \longrightarrow \mathcal{G}$ is \emph{injective} (resp. \emph{surjective}) if each $u_\alpha$ is \emph{injective} (resp. \emph{surjective}).
\end{defi}

\begin{rema}
    The compatibility condition with $\alpha = \beta$ ensures that $\mathcal{F}_\alpha$ is a $G_\alpha$-sheaf, where $G_\alpha := \mathrm{Gal}\; (\eta_\alpha)$. When the index $\alpha$ is not precised, we will prefer notation $G_\eta$ to denote the Galois group of the morphism of a l.u.c. $(U, V, \eta)$, and $\mathcal{F}_\eta$ to denote the local sheaf defining $\mathcal{F}$ in this chart.
\end{rema}

\begin{defi}
    Let $(X, \Delta), (Y, \Delta')$ be two pairs endowed with orbistructures $\mathcal{C} = \lbrace (U_\alpha, V_\alpha, \eta_\alpha) \rbrace_\alpha$ and $\mathcal{D} = \lbrace (T_{\alpha'}, W_{\alpha'}, \mu_{\alpha'}) \rbrace_\alpha$ respectively. Let $\mathcal{F}$ be an orbisheaf on $(X, \Delta)$ and let $\pi \colon (Y, \Delta') \longrightarrow (X, \Delta)$ be an orbifold morphism. 
    
    \emph{The pullback of $\mathcal{F}$ by $\pi$}, denoted $\pi^* \mathcal{F}$, is the orbisheaf defined by the collection $(\pi_{\alpha\alpha'}^* \mathcal{F}_\alpha)$. If $ (\varphi_{\alpha\beta})$ are the transition maps of $\mathcal{F}$, the transition maps of $\pi^* \mathcal{F}$ are $(\pi_{\alpha\beta, \alpha'\beta'}^{[*]} \varphi_{\alpha\beta})$ where $\pi_{\alpha\beta, \alpha'\beta'}^*$ has been defined in Remark~\ref{remarkoforbifoldmorphism}, and they satisfy cocycle conditions on triple intersections. We will denote by $\pi^{[*]} \mathcal{F} := (\pi^* \mathcal{F})^{\vee\vee}$.
\end{defi}

\begin{exem}
     When $(X, \Delta)$ is an orbifold, the \textit{orbifold cotangent sheaf} $\Omega^1_{(X, \Delta)}$ is given by the collection $(\Omega^1_{V_\alpha})_\alpha$, and the transition functions are given by the composition of the following isomorphisms (because in this case the maps $g_{\alpha\beta}$ are \'etale):
     $$
     p_{\alpha\beta} \colon g_{\beta\alpha}^* \Omega^1_{V_\beta} \overset{\sim}{\longrightarrow} \Omega^1_{V_{\alpha\beta}} \overset{\sim}{\longrightarrow} g_{\alpha\beta}^* \Omega^1_{V_\alpha}.
     $$
     Its dual is the \textit{orbifold tangent sheaf} $\mathcal{T}_{(X, \Delta)}$.
\end{exem}

\begin{exem}
    Let $(X, \Delta)$ be a pair endowed with an orbistructure, and let $\mathcal{F}$ be a coherent sheaf on $X$. The collection $(\eta_\alpha^* \mathcal{F}_{| U_\alpha})_\alpha$ defines an orbisheaf with natural transition maps $\mathrm{nat}_{\alpha\beta} \colon g_{\beta\alpha}^{[*]} \eta_\beta^* \mathcal{F} \overset{\sim}{\longrightarrow} g_{\alpha\beta}^{[*]} \eta_\alpha^* \mathcal{F}$ coming from $\eta_\alpha \circ g_{\alpha\beta} = \eta_\beta \circ g_{\beta\alpha}$. We denote by $\mathcal{F}^\orb$ this orbisheaf. We will denote $\mathcal{F}^{[\orb]} = {\mathcal{F}^\orb}^{\vee\vee}$.
\end{exem}

\subsubsection{Canonical extension of an orbifold}

In order to define the canonical extension associated to an orbifold $(X, \Delta)$, we need to give a correspondence between classes of extensions of orbisheaves and first extension group.

\begin{prop}\label{extensionorbisheaves}
    Let $\mathcal{S}$, $\mathcal{Q}$ be two orbisheaves on an orbifold $(X, \Delta)$, with transition morphisms denoted by $\varphi_{\alpha\beta}$ and $\psi_{\alpha\beta}$ respectively.
    
    Let $\Omega := \left( \omega_{\alpha\beta} \colon g_{\beta \alpha}^* \mathcal{Q}_\beta \longrightarrow g_{\alpha \beta}^* \mathcal{S}_\alpha \right)_{\alpha,\beta}$ be a collection of $\mathcal{O}_{V_{\alpha\beta}}$-modules morphisms, and $h_{\alpha\beta,\gamma} \colon V_{\alpha\beta\gamma} \longrightarrow V_{\alpha\beta}$ be the natural morphisms induced by triple intersections. If $\Omega$ satisfies the following cocycle conditions:
    $$
    h_{\alpha\beta,\gamma}^* \varphi_{\alpha\beta} \circ h_{\beta\gamma, \alpha }^* \omega_{\beta\gamma} + h_{\alpha\beta,\gamma}^* \omega_{\alpha\beta} \circ h_{\beta\gamma, \alpha }^* \psi_{\beta\gamma} = h_{\gamma\alpha, \beta}^* \omega_{\alpha\gamma},
    $$
    then, it induces an extension $\mathcal{F}$ of $\mathcal{Q}$ by $\mathcal{S}$, it means a short exact sequence of orbisheaves:
    $$
    0 \longrightarrow \mathcal{S} \longrightarrow \mathcal{F} \longrightarrow \mathcal{Q} \longrightarrow 0.
    $$
\end{prop}

\begin{proof}
    Let $\theta_{\alpha\beta} \colon g_{\beta\alpha}^* (\mathcal{S}_{\beta} \oplus \mathcal{Q}_{\beta} ) \longrightarrow g_{\alpha\beta}^*(\mathcal{S}_{\alpha} \oplus \mathcal{Q}_{\alpha} )$  be the $\mathcal{O}_{V_{\alpha\beta}}$-modules morphism defined by:
    $$\theta_{\alpha\beta} :=
    \left[
    \begin{array}{c|ccc}
       \varphi_{\alpha\beta}  & & \omega_{\alpha\beta} & \\
       \hline
       &&& \\
       0 && \psi_{\alpha\beta} & \\
       \\         
    \end{array}
    \right]
    $$
    The cocycle conditions satisfied by $(\varphi_{\alpha\beta})_{\alpha\beta}$, $(\psi_{\alpha\beta})_{\alpha\beta}$ $(\omega_{\alpha\beta})_{\alpha\beta}$ induce those on $\theta_{\alpha\beta}$. Hence, they determine an orbisheaf $\mathcal{F}$, given by the collection $(\mathcal{S}_{\alpha} \oplus \mathcal{Q}_{\alpha})_\alpha$ and the transition morphisms $(\theta_{\alpha\beta})_{\alpha\beta}$. Moreover, the injections $\mathcal{S}_\alpha \hookrightarrow \mathcal{S}_\alpha \oplus \mathcal{Q}_\alpha$ and the projections $\mathcal{S}_\alpha \oplus \mathcal{Q}_\alpha \longrightarrow \mathcal{Q}_\alpha$ are compatible with transition morphisms.
\end{proof}

\begin{exem}\label{deforbifoldextension}
    Let $\mathcal{Q} = \mathcal{T}_{(X,\Delta)}$ and $\mathcal{S} = \mathcal{O}_X$.\\
    We assume that the orbifold structure is such that the $V_\alpha$ are trivializing open sets for the rank 1 orbibundle $\Omega^n_{(X, \Delta)}$. Let $\varphi_{\alpha\beta}$ be the transition morphisms of this orbibundle. These are automorphisms of $\mathcal{O}_{V_{\alpha\beta}}$. The image of $1$ gives a section $\Phi_{\alpha\beta} \in \mathcal{O}_{V_{\alpha\beta}}(V_{\alpha\beta})^\times$. We have the cocycle condition: $h_{\alpha\beta, \gamma}^* \Phi_{\alpha\beta} h_{\beta\gamma, \alpha}^* \Phi_{\beta\gamma} = h_{\gamma\alpha, \beta}^* \Phi_{\alpha\gamma}$. By taking the derivative of its logarithm, we get: 
    $$
    h_{\alpha\beta, \gamma}^* \left( \frac{ d\Phi_{\alpha\beta}}{\Phi_{\alpha\beta}} \right) + h_{\beta\gamma, \alpha}^* \left( \frac{d\Phi_{\beta\gamma}}{\Phi_{\beta\gamma}} \right) = h_{\gamma\alpha, \beta}^* \left( \frac{d\Phi_{\alpha\gamma}}{\Phi_{\alpha\gamma}} \right).
    $$
    We obtain a collection of $1$-forms $\Omega := \left( \frac{\mathrm{d}\Phi_{\alpha\beta}}{\Phi_{\alpha\beta}} \right)$ on $V_{\alpha\beta}$. With the identification $\Omega^1_{V_{\alpha\beta}} = \mathcal{T}_{V_{\alpha\beta}}^\vee \simeq g_{\alpha\beta}^*\mathcal{T}_{V_{\alpha}}^\vee$, we denote $\omega_{\alpha\beta} := \frac{\mathrm{d}\Phi_{\alpha\beta}}{\Phi_{\alpha\beta}}$, so that the collection $\Omega$ consists of morphisms $\omega_{\alpha\beta} \colon g_{\alpha\beta}^* \mathcal{T}_{V_\alpha} \longrightarrow \mathcal{O}_{V_{\alpha\beta}}$ satisfying the closed \v{C}ech cocycle relation of Proposition~\ref{extensionorbisheaves}. A consequence of this proposition is that it provides us with an extension $\mathcal{E}_{(X, \Delta)}$ of $\mathcal{T}_{(X, \Delta)}$ by $\mathcal{O}_X$, called the \textit{canonical extension of $(X, \Delta)$}. When $\Delta = 0$ and $X$ is smooth, this construction produces the classical canonical extension.\\
\end{exem}

\subsection{Flatness and developability}

\subsubsection{Orbi-locally freeness and sheaf of invariants}

In order to investigate to what extent an orbibundle is obtained as a pullback of a bundle on $X$, we have the following proposition.

\begin{prop}\label{orbisheafgeneratedinvariant}
    Let $(X, \Delta)$ be a pair endowed with an orbistructure, and let $\mathcal{F}$ be an orbibundle of rank $r$ defined on $(X, \Delta)$. Let $x$ be a point of $X$. The following assertions are equivalent:\\
    $(i)$ There exists a neighbourhood $U_0$ around $x$ and an isomorphism of orbisheaves $\mathcal{F}_{| U_0} \simeq \mathcal{O}_{U_0}^{\oplus r}$.\\
    $(ii)$ There exists a neighbourhood $U_0$ of $x$ such that for all l.u.c. $(U,V,\eta)$ around $x$, ${\mathcal{F}_{\eta}}_{|\eta^{-1}(U_0)}$ is a free $\mathcal{O}_{\eta^{-1}(U_0)}$-module of rank $r$, and there exists a basis composed by $G_\eta$-invariant sections.\\
    $(iii)$ There exists a neighbourhood $U_0$ and a l.u.c. $(U,V,\eta)$ around $x$ such that ${\mathcal{F}_{\eta}}_{|\eta^{-1}(U_0)}$ is a free $\mathcal{O}_{\eta^{-1}(U_0)}$-module of rank $r$, and there exists a basis composed by $G_\eta$-invariant sections.
\end{prop}

\begin{defi}
    With notation as above, we say that \emph{$\mathcal{F}$ is orbi-locally free around $x$} if one of the assertions of Proposition~\ref{orbisheafgeneratedinvariant} is satisfied. If $\mathcal{F}$ is orbi-locally free around each point, we say that \emph{$\mathcal{F}$ is orbi-locally free}.
\end{defi}

\begin{defi}\label{sheafofinvariants}
    Let $(X, \Delta)$ be a pair endowed with an orbistructure, and $\mathcal{F}$ a reflexive orbisheaf. We can define a sheaf on $X$ in the following way. For each l.u.c. $(U_\alpha, V_\alpha, \eta_\alpha)$, there is a sheaf $\mathcal{F}_\alpha^{\inv}$ on $U_\alpha$ defined by:
    $$
    \mathcal{F}_\alpha^{\inv} \colon U \longmapsto H^0( \eta_\alpha^{-1}(U), \mathcal{F}_\alpha)^{G_\alpha},
    $$
    that associates to each open subset $U$ of $U_\alpha$ the set of $G_\alpha$-invariant sections of $\mathcal{F}_\alpha$ defined on $\eta_\alpha^{-1}(U)$. This sheaf is reflexive by \cite[Lemma.~A.4]{GKKP11}.

    There are morphisms $\theta_{\alpha\beta} \colon {\mathcal{F}_\beta^{\inv}}_{| U_\alpha \cap U_\beta} \longrightarrow {\mathcal{F}_\alpha^{\inv}}_{| U_\alpha \cap U_\beta}$ induced by transition maps of the orbisheaf $\mathcal{F}$ (here we use reflexivity of $\mathcal{F}_\alpha$ to extend a section of $\mathcal{F}_\alpha$ obtained by descent of a $G_\alpha$-invariant section on $g_{\alpha\beta}^{[*]} \mathcal{F}_\alpha$ away from the non-\'etale locus of $g_{\alpha\beta}$). They satisfy cocycle conditions, and it follows that the sheaves $\mathcal{F}_\alpha^{\inv}$ glue together to define a reflexive sheaf $\mathcal{F}^{\inv}$ on $X$.\\

    Moreover, if $u \colon \mathcal{F} \longrightarrow \mathcal{G}$ is a morphism of reflexive orbisheaves, we can construct a morphism $u^{\inv} \colon \mathcal{F}^\inv \longrightarrow \mathcal{G}^\inv $ as following. For each $\alpha$, there is a morphism $u_\alpha^{\inv} \colon \mathcal{F}_\alpha^{\inv} \longrightarrow \mathcal{G}_\alpha^{\inv}$ induced by $u_\alpha$. Indeed, the morphism $u_\alpha$ sends $G_\alpha$-invariant sections of $\mathcal{F}_\alpha$ on invariant sections of $\mathcal{G}_\alpha$ because it is $G_\alpha$-equivariant. These morphisms are compatible with cocycles of $\mathcal{F}^\mathrm{inv}$ and $\mathcal{G}^{\inv}$, because the morphisms $u_\alpha$ are compatible with transition morphisms of the orbisheaves $\mathcal{F}$ and $\mathcal{G}$. They glue together to compose a global morphism $u^{\inv} \colon \mathcal{F}^{\inv} \longrightarrow \mathcal{G}^{\inv}$. This construction is compatible with composition of orbisheaves and preserves identity morphisms. Hence, it is a functor from the category of orbisheaves to the category of sheaves.
\end{defi}

\begin{prop}\label{factsinvariantsheaves}
    Let $(X, \Delta)$ be a pair endowed with an orbistructure. We have the following statements.\\
\noindent
\underline{A.} If $\mathcal{A}$ is a reflexive sheaf on $X$, then we have a natural isomorphism:  
$$ \mathcal{A} \simeq (\mathcal{A}^{[\orb]})^{\inv}.
$$

\noindent
\underline{B.} If $\Delta = 0$ and $\mathcal{F}$ is a reflexive orbisheaf, then:
$$
\mathcal{F} \simeq \left( \mathcal{F}^\inv \right)^{[\orb]}.
$$

\noindent
\underline{C.} The orbibundle $\mathcal{F}$ is orbi-locally free iff $\mathcal{F}^{\inv}$ is locally free and there is an isomorphism of orbisheaves: $\mathcal{F} \simeq (\mathcal{F}^\inv)^\orb$.

\end{prop}

\begin{proof}
    For the point A, let $U \hookrightarrow U_\alpha$ be an open subset and $(U_\alpha, V_\alpha, \eta_\alpha)$ is a l.u.c., the isomorphism is locally given by:
    $$
    s \in \mathcal{A}_{| U_\alpha}(U) \longmapsto s \circ \eta_\alpha \in H^0(\eta_\alpha^{-1}(U), \eta_\alpha^*\mathcal{A})^{G_\alpha}.
    $$
    First, we show that it is an isomorphism when $\mathcal{A}$ is locally free.
   
    It is injective as $\eta_\alpha$ is onto $U_\alpha$. Let us show that this morphism is surjective. 
    It is true if $\mathcal{A} = \mathcal{O}_X$ by definition of functions on a quotient space. Now, if $\mathcal{A}_{| U}$ is isomorphic to $\mathcal{O}_{U}^{\oplus r}$ via a trivialization $s$, the precomposition $\widetilde{s} := s \circ \eta_\alpha$ provides a trivialization for $\eta_\alpha^* \mathcal{A}_{| U}$ by invariant sections. It follows by the case of the trivial bundle that a section $\displaystyle t := \sum f_i \widetilde{s_i}$ which is $G_\alpha$-invariant yields functions $g_1, \ldots, g_r$ in $\mathcal{O}_U$ such that $g_i = f_i \circ \eta_\alpha$. Hence, the section $t$ is obtained by the precomposition by $\eta_\alpha$ of the section $\displaystyle \sum g_i s_i$. It implies that the morphism is onto.

    Now, we assume that $\mathcal{A}$ is reflexive. Because $\mathcal{A}_{| U_\alpha}$ and $\mathcal{A}_\alpha^\inv$ are reflexive and are isomorphic in codimension 1 (since $\mathcal{A}$ is locally free in codimension 1), the isomorphism extends to $U_\alpha$.
    Then, the isomorphism $g_{\alpha\beta}^{[*]} \eta_\alpha^{[*]} \mathcal{A} = g_{\beta\alpha}^{[*]} \eta_\beta^{[*]} \mathcal{A}$ is given by the identity of $\mathcal{A}$, and it follows that this isomorphism commutes with the transition morphism $\theta_{\alpha\beta}$ defined in Definition~\ref{sheafofinvariants} on intersections $U_\alpha \cap U_\beta$. Finally, $\mathcal{A}$ is naturally isomorphic to the sheaf of invariant sections of its induced orbisheaf.\\

    \noindent
    For the point B, if $(U_\alpha, V_\alpha, \eta_\alpha)$ is a l.u.c., we observe by using the compatibility condition between $\mathcal{F}_\alpha$ and a trivializing open subset of $\mathcal{F}^\inv$ on $X_\reg$ that $\mathcal{F}_{\alpha \, |\eta_\alpha^{-1}(X_\reg)} \simeq \eta_\alpha^* \mathcal{F^\inv}_{|X_\reg} $. We have an isomorphism $\mathcal{F}_\alpha \simeq \eta_\alpha^{[*]} \mathcal{F}^\inv$ because both sheaves are reflexive and coincide on $\eta_\alpha^{-1}(X_\reg)$ which is a big open subset. The compatibility condition with transition maps follows from that on triple intersections with trivializing open subsets of $\mathcal{F}^\inv$ on $X_\reg$.\\

    \noindent
    For the point C, if $\mathcal{F}$ is orbi-locally free, then the sheaves $\mathcal{F}_\alpha^\inv$ are locally free by $(i)$ in Proposition~\ref{orbisheafgeneratedinvariant}. Hence, the sheaf $\mathcal{F}^\inv$ is locally free. Moreover, we have $\eta_\alpha^* \mathcal{F}_\alpha^\inv \simeq \mathcal{F}_\alpha$ (as a consequence of \cite[Thm.~4.2.15]{HL10}, there exists a locally free sheaf $\mathcal{G}_\alpha$ such that $\mathcal{F}_\alpha \simeq \eta_\alpha^* \mathcal{G}_\alpha$, and previous arguments show that $\mathcal{F}_\alpha^\inv \simeq (\eta_\alpha^* \mathcal{G}_\alpha)^\inv = \mathcal{G}_\alpha$.). This isomorphism is compatible with the transition maps of $\mathcal{F}$ and that of $(\mathcal{F}^\inv)^\orb$ (these are $g_{\alpha\beta}^* \eta_\alpha^* \theta_{\alpha\beta}$, see Definition~\ref{sheafofinvariants} for notation), by taking local invariant frames. For the converse, as $\mathcal{F}^\inv$ has local trivializations, we get that the orbisheaf $\mathcal{F}$ is locally isomorphic to the trivial $\mathcal{O}_X^{\oplus r}$, that exactly means that $\mathcal{F}$ is orbi-locally free.
\end{proof}

\subsubsection{Description of orbifolds with an orbi-locally free tangent bundle}

In order to give a criterion of developability of a pair, we give the following two results that describe orbifolds whose orbifold tangent sheaf is orbi-locally free:

\begin{lemm}\label{delta0}
    Let $(X, \Delta)$ be an orbifold. Assume that $\mathcal{T}_{(X, \Delta)}$ is orbi-locally free. Then, $\Delta = 0$ and $X$ is smooth.
\end{lemm}

\begin{proof}
    The orbi-locally freeness of $\mathcal{T}_{(X, \Delta)}$ is equivalent to that of $\Omega^1_{(X, \Delta)}$, because both are dual orbisheaves. Hence, we can assume that $\Omega^1_{(X, \Delta)}$ is orbi-locally free. 
    We denote: $\displaystyle \Delta =  \sum_i \left( 1 - \frac{1}{m_i} \right) \Delta_i$. At a general point $x \in \Delta_i$, we have an adapted l.u.c.$(U_\alpha, V_\alpha, \eta_\alpha)$ above $x$ of degree $m_i$ (it means that its normal form is $\eta_\alpha(z_1, \underline{z}) = (z_1^{m_i}, \underline{z})$). We consider $(s_1, \ldots, s_{n})$ a system of generators of $\Omega^1_{V_\alpha, y}$ at a point $y \in \eta_\alpha^{-1}(x)$, invariant under the action of $\Z / (m_i)$. Let $\displaystyle \mathrm{d}z_1 := \sum_k f_k s_k$, with $f_{k_0}(0) \neq 0$ for some $1 \leq k_0 \leq n$.\\
    The action of the class $\left[ 1 \right] \in \Z / (m_i) $ on $\mathrm{d}z_1$ is: $\left[ 1 \right]. \mathrm{d}z_1 = \zeta^{-1} \mathrm{d}z_1$, where $\displaystyle \zeta = e^{\sqrt{-1}\frac{2\pi}{m_i}}$. On the other hand, it is: $\displaystyle \left[ 1 \right]. \mathrm{d}z_1 = \sum f_k(\zeta^{-1} z_1, \underline{z}) s_k$. By evaluating at 0, we have $\zeta f_{k_0}(0) = f_{k_0}(0)$, then $\zeta = 1$, hence $m_i = 1$. It is true for each $i$, so $\Delta = 0$.

    Now, we have $\mathcal{T}_{(X, 0)}$ is orbi-locally free because so is $\Omega^1_{(X, 0)}$. By the point C of Proposition~\ref{factsinvariantsheaves}, we get that $\mathcal{T}_{(X, 0)}^\inv$ is locally free, and it coincides on $X_\reg$ with $\mathcal{T}_X$. Hence, they are isomorphic because both are reflexive, and by \cite[Thm.~1.1]{Dru14}, $X$ is smooth.
\end{proof}

\begin{coro}\label{corodevelopability}
    Let $(X, \Delta)$ be an orbifold. Assume that there exists a surjective morphism of orbibundles: $ \xymatrix{\mathcal{O}_X^{\oplus r} \ar@{->>}[r] & \mathcal{T}_{(X, \Delta)}
    }$.
    Then, $\Delta = 0$ and $X$ is smooth.
\end{coro}

\begin{proof}
    By Lemma~\ref{delta0}, it suffices to show that $\mathcal{T}_{(X, \Delta)}$ is orbi-locally free. By Proposition~\ref{orbisheafgeneratedinvariant}, it suffices to construct for any l.u.c.$(U_\alpha, V_\alpha, \eta_\alpha)$ and $y \in V_\alpha$ a basis of $G_\alpha$-invariant sections for $\mathcal{T}_{V_\alpha}$ on a $G_\alpha$-invariant open neighborhood $V$ of $y$. By assumption, we have a morphism of $G_\alpha$-sheaves: $\xymatrix{
    \mathcal{O}_{V_\alpha}^{\oplus r} \ar@{->>}[r] & \mathcal{T}_{V_\alpha}
    }$.
    We denote by $t_1, \ldots, t_r$ the image of the canonical basis $e_1, \ldots, e_r$ by this morphism. Hence, these sections are $G_\alpha$-invariant, and they generate the tangent bundle. It follows that there exists a neighbourhood of $y$ denoted $V^\circ \hookrightarrow V_\alpha$ such that we can extract $n$ sections (for example $t_1, \ldots t_n$) that form a local frame of $\mathcal{T}_{V_\alpha}$. Moreover, if $g \in G_\alpha$, then the family $(g(t_1), \ldots, g(t_n))$ forms a trivialization of $\mathcal{T}_{V_\alpha}$ around $g(y)$. Because the sections $t_i$ are $G_\alpha$-invariant, we infer that $(t_1, \ldots, t_n)$ provides a local frame on a $G_\alpha$-invariant open neighborhood $V = \bigcup g \cdot V^\circ$, providing the sought trivialization of ${\mathcal{T}_{V_\alpha}}_{| V}$ by $G_\alpha$-invariant sections. 
\end{proof}

\begin{exem}
    We give the example of a cone over a smooth conic: the group $\lbrace \pm 1 \rbrace$ acts on the affine plane $\C^2$, the quotient space $X$ is isomorphic to $(w^2 - uv = 0) \subseteq \C^3$ (we denote by $(u,v,w)$ the coordinates in the target space) and the quotient map is given by $\pi \colon (x, y) \in \C^2 \longmapsto (x^2, y^2, xy) \in X $. This map is quasi-étale, and the l.u.c. $(X, \C^2, \pi)$ provides an orbi-étale orbistructure for the pair $(X, 0)$. We have $\Delta = 0$ and $X$ is singular at 0. One can show that there does not exist a surjective map $\xymatrix{\mathcal{O}_{\C^2}^{\oplus 2} \ar@{->>}[r] & \mathcal{T}_{\C^2} }$ which is $G$-equivariant and surjective. Indeed, if such a map exists, then the image of the canonical basis provides vector fields that generate $\mathcal{T}_{\C^2}$, and they are invariant by the action of $\lbrace \pm 1 \rbrace$. Hence, such vector fields have their germs at 0 of the form:
    $$
    f(x^2, y^2) x \frac{\partial}{\partial x} + g(x^2, y^2) y \frac{\partial}{\partial x} + h(x^2, y^2) x \frac{\partial}{\partial y} + k(x^2, y^2) y \frac{\partial}{\partial y}, 
    $$
    where $f, g, h, k \in \mathcal{O}_{\C^2, 0}$. It belongs to $\mathfrak{m} \mathcal{T}_{\C^2, 0}$, where $\mathfrak{m}$ is the maximal ideal of $\mathcal{O}_{\C^2, 0}$. Hence, it can not generate $\mathcal{T}_{\C^2}$ at $0$.

    Moreover, this example illustrates that, in the case $\Delta = 0$, the functors $(\cdot)^{[\orb]}$ and $(\cdot)^\inv$ are essentially surjective by points A and B of Proposition~\ref{factsinvariantsheaves}, but they are not equivalences of categories. Indeed, let us show that the functor $(\cdot)^{[\orb]} $ is not exact, because it does not preserve surjective morphisms. Let us consider a surjective morphism $\xymatrix{ \mathcal{O}_X^{\oplus r} \ar@{->>}[r] & \mathcal{T}_X }$. Its reflexive pullback by $\pi$ provides a family of vector fields of $\C^2$ which are invariant by the action of $\lbrace \pm 1 \rbrace$. The resulting morphism is no more surjective, as its image is included in $\mathfrak{m}\mathcal{T}_{\C^2, 0}$.
\end{exem}

\subsubsection{Flatness and developability}

We will use classical tools in orbifold differential geometry (orbifold hermitian metrics and connections, orbifold Chern classes (the $i$-th orbifold Chern class of an orbibundle $\mathcal{F}$ is denoted by $c_i^\orb(\mathcal{F})$), orbifold de Rham cohomology, \ldots) which are detailed in \cite[Sect.~2]{Bla96}.

\noindent
We have the following criterion of developability of an orbifold.

\begin{prop}[Developability criterion for an orbifold]\label{Developability}
    Let $(X, \Delta)$ be an orbifold and let $\mathcal{F}$ be a flat orbibundle of rank $r$ on $(X, \Delta)$. Let us assume that there exists a surjective morphism of orbibundles: $\xymatrix{\mathcal{F} \ar@{->>}[r] &  \mathcal{T}_{(X, \Delta)}}$. Then, the universal cover $(\widetilde{X}_\Delta, \widetilde{\Delta})$ satisfies $\widetilde{\Delta} = 0$ and $\widetilde{X}_\Delta$ is smooth.
\end{prop}

\begin{proof}
    Let us consider the universal cover map: $\pi \colon (\widetilde{X}_\Delta, \widetilde{\Delta}) \longrightarrow (X, \Delta)$. The pullback $\pi^* \mathcal{F}$ remains flat on $(\widetilde{X}_\Delta, \widetilde{\Delta})$. By the Riemann--Hilbert correspondence in the orbifold setting that states an equivalence between representations of $\pi_1^\orb(X, \Delta)$ and flat orbibundles on $(X, \Delta)$ (see \cite[Thm.~2.35]{SY22}), we get that it is a trivial bundle $\mathcal{O}_{\widetilde{X}_\Delta}^{\oplus r}$. Moreover, we have $\pi^* \mathcal{T}_{(X, \Delta)} = \mathcal{T}_{(\widetilde{X}_\Delta, \widetilde{\Delta})}$. Hence, we have a surjective morphism 
    $
    \xymatrix{
    \mathcal{O}_{\widetilde{X}_\Delta}^{\oplus r} \ar@{->>}[r] & \mathcal{T}_{(\widetilde{X}_\Delta, \widetilde{\Delta})}
    }.
    $
    We are in position to apply Corollary~\ref{corodevelopability}, and we get $\widetilde{\Delta} = 0$ and $\widetilde{X}_\Delta$ is smooth.
\end{proof}

\begin{coro}\label{Endcanflat}
    Let $(X, \Delta)$ be an orbifold such that the orbibundle $\End(\mathcal{E}_{(X, \Delta)})$ is flat. Then, $\widetilde{X}_\Delta$ is smooth and $\widetilde{\Delta} = 0$.
\end{coro}

\begin{proof}
    We have a surjective morphism $\xymatrix{\mathcal{E}_{(X, \Delta)}^\vee \ar@{->>}[r] & \mathcal{O}_X } $ obtained by applying the dual functor to the short exact sequence defining $\mathcal{E}_{(X, \Delta)}$. By tensoring with $\mathcal{E}_{(X, \Delta)}$, we get a surjective morphism $\xymatrix{\End(\mathcal{E}_{(X, \Delta)}) \ar@{->>}[r] & \mathcal{E}_{(X, \Delta) } } $. These operations preserve surjective maps as we deal with orbibundles. By composition, we get a surjective morphism $\xymatrix{\End(\mathcal{E}_{(X, \Delta)}) \ar@{->>}[r] & \mathcal{T}_{(X, \Delta) } }$, and we conclude by using Proposition~\ref{Developability}.
\end{proof}

\begin{rema}
    Let $(X, \Delta, \omega)$ be a compact Kähler orbifold. One can define slope-stability (resp. semistability, polystability) with respect to $\left[ \omega \right]$ for torsion free orbisheaves on orbifolds (see \cite[Def.~2.29]{GT22}). In this setting, the Kobayashi--Hitchin correspondence holds, so that polystable orbibundles can be endowed with Hermite--Einstein metrics with respect to $\omega$ by \cite[Theorem~4.2]{ES18}.
    
    Now, let us assume that the Kähler orbifold $(X, \Delta, \omega)$ satisfies the Miyaoka--Yau equality and the canonical extension is polystable with respect to $\left[ \omega \right] $. Then, the equality case in the Kobayashi--Lübke inequality applied to $\mathcal{E}_{(X, \Delta)}$ (see \cite[Thm.~4.4.7]{Kob87}) shows that it is projectively flat. It follows by adapting local computation of \cite[Prop.~1.2.9]{Kob87} that $\End(\mathcal{E}_{(X, \Delta)})$ is flat, and the universal cover $(\widetilde{X}_\Delta, \widetilde{\Delta})$ satisfies $\widetilde{X}_\Delta$ is smooth and $\widetilde{\Delta} = 0$ as a consequence of Corollary \ref{Endcanflat}. Finally, by pulling back $\End(\mathcal{E}_{(X, \Delta)}) $ on the universal cover, we get that $\End(\mathcal{E}_{\widetilde{X}_\Delta})$ is flat. Hence, $\mathcal{E}_{\widetilde{X}_\Delta}$ is projectively flat and simply connected, so that there exists a line bundle $\mathcal{L}$ such that $\mathcal{E}_{\widetilde{X}_\Delta} = \mathcal{L}^{\oplus n+1}$. Finally, we have $c_1(X) = (n+1)c_1(L)$ so that $\widetilde{X}_\Delta = \mathbb{P}^n$ by \cite[Cor. of Thm.~1.1]{KO73}.
\end{rema}

\section{Klt pairs}

\begin{defi}\label{Defkltpair}
    Let $(X, \Delta)$ be a pair. We say that the pair is \emph{klt} if:
    \begin{itemize}
        \item the Weil $\Q$-divisor $K_X + \Delta$ is $\Q$-Cartier.
        \item the discrepancies are greater than -1. More precisely, for a log resolution $\mu \colon \widehat{X} \longrightarrow X$ of $(X, \Delta)$, we set:
        $$
        K_{\widehat{X}} = \mu^* (K_X + \Delta) + \sum_i a_i E_i,
        $$
        where $\displaystyle \sum_i E_i = \mathrm{Exc}(\mu) + \mu^{-1}_* \Delta$ is the sum of the exceptional divisor of $\mu$ denoted by $\mathrm{Exc}(\mu)$, and the support of the strict transform of $\Delta$ denoted by $\mu^{-1}_* \Delta$. Then, we have $a_i > -1$ for each $i$.
    \end{itemize}
    We say that the pair is \emph{log Fano} if $(X, \Delta)$ is klt, $X$ is projective and $-(K_X + \Delta)$ is $\Q$-ample.\\
\end{defi}

There are different ways to define orbistructures on $(X, \Delta)$. The first construction is an orbi-\'etale one, due to \cite[Lemma~14]{CGG24}. Let us recall its construction.

\begin{rema}\label{firstorbistructurekltpair}
The $\Q$-divisor $K_X + \Delta$ is $\Q$-Cartier, hence, for an integer $m \gg 1$, we can cover $X$ by open sets $U_\alpha$ such that $ m(K_{U_\alpha} + \Delta_{| U_\alpha})$ is principal. The cyclic covering map of degree $m$ branched along this principal divisor provides an orbi-\'etale map $\eta_\alpha \colon V_\alpha \longrightarrow U_\alpha$. One can show that the morphisms are compatible on intersections, by studying local forms of maps above general points of $\Delta$. The collection $(U_\alpha, V_\alpha, \eta_\alpha)$ provides us with an orbi-\'etale orbistructure.
\end{rema}

We can refine this orbistructure to get a smooth one in codimension 2.

The second construction yields a strictly adapted orbistructure. It is a consequence of the following proposition.

\begin{prop}\label{elimdelta}{\cite[Prop.~13]{CGG24}}
Let $X$ be a normal projective variety, and $\Delta$ a Weil $\Q$-divisor with standard coefficients such that $K_X + \Delta$ is $\Q$-Cartier. 
\\
Then, there exists a very ample line bundle $L$ on $X$ such that for general $H \in |L|$, there exists a cyclic Galois morphism $f \colon Y \longrightarrow X$ (we denote by $N$ its degree) such that $f$ is orbi-\'etale for $\displaystyle \left( X, \Delta + \left( 1 - \frac{1}{N} \right) H \right)$. Moreover, if $(X, \Delta)$ is klt, so are $\displaystyle \left( X, \Delta + \left( 1 - \frac{1}{N}\right) H \right)$ and $(Y, 0)$.
\end{prop}

This second orbistructure consists of global morphisms, in the sense that maps $f \colon Y \longrightarrow X$ are onto and have a good behaviour with respect to $\Delta$. In order to prove that a log Fano pair $(X, \Delta)$ is orbifold under the assumptions of the main theorem, we will show that the varieties $Y$ are smooth away from a divisor, following the strategy of \cite[Sect.~5.3, Direction (1.3.1) $\Longrightarrow$ (1.3.2)]{GKP22}.

\subsection{Adapted canonical extension}

Let $(X, \Delta)$ be a klt pair, and $f \colon Y \longrightarrow X$ a $\Delta$-strictly adapted morphism. We recall classical constructions of adapted sheaves initiated by Miyaoka. See \cite[Section.~3]{CKT16} for details.

\begin{defi}
    Let $(X, \Delta)$ be a pair and $f \colon Y \longrightarrow X$ a strictly $\Delta$-adapted morphism. Let $X_0 \subseteq X $ and $\iota \colon Y_0 \hookrightarrow Y$ be the maximal open subsets where $f$ is good in the sense of \cite[Def.~3.5]{CKT16}. The \emph{sheaf of adapted differentials}, denoted by $\Omega^{[1]}_{X, \Delta, f}$, is the following sheaf:
    $$
    \Omega^{[1]}_{X, \Delta, f} = \iota_*  \left[ \left( \mathrm{im} \left( f^* \Omega^1_{X_0} \longrightarrow \Omega^1_{Y_0} \right) \otimes \mathcal{O}(f^* \Delta) \right)\cap \Omega^1_{Y_0} \right],
    $$

    where $\Omega^1_{X_0}$ (resp. $\Omega^1_{Y_0}$) denotes the cotangent sheaf on $X_0$ (resp. on $Y_0$).
    The \emph{adapted tangent sheaf}, denoted by $\mathcal{T}_{X, \Delta, f}$, is its dual $\mathcal{T}_{X, \Delta, f} := {\Omega^{[1]}_{X, \Delta, f}}^\vee$.
\end{defi}

Now, we construct the adapted canonical extension $\mathcal{E}_{X, \Delta, f}$. It is a reflexive sheaf obtained as an extension of $\mathcal{T}_{X, \Delta, f}$ by $\mathcal{O}_Y$ as follows.

Let $m \gg 1$ be such that $-m(K_X + \Delta)$ is Cartier. It provides a \v{C}ech closed cocycle $(h_{\alpha\beta})_{\alpha\beta}$ on $X$. The pullback of the derivative of its logarithm $\left( f^* \frac{\mathrm{d}\, h_{\alpha\beta} }{h_{\alpha\beta}}  \right)$ is a closed cocyle with values in: $\mathrm{im} \left( f^* \Omega^1_X \longrightarrow \Omega^1_Y \right)$. We have a natural morphism from this subsheaf of $\Omega^1_Y$ into $ \Omega^{[1]}_{X, \Delta, f}$. Hence, it gives rise to a class $H^1(Y, \Omega^{[1]}_{X, \Delta, f})$, we denote it by $m f^* c_1(X, \Delta)$. We consider the extension induced by $f^* c_1(X, \Delta)$:
$$
0 \longrightarrow \Omega^{[1]}_{X, \Delta, f} \longrightarrow \mathcal{W}_{X, \Delta, f} \longrightarrow \mathcal{O}_Y \longrightarrow 0.
$$
We use the functor $\mathcal{H}om(-, \mathcal{O}_Y)$ which produces a short exact sequence because $\mathcal{E}xt^1(\mathcal{O}_Y, \mathcal{O}_Y) = 0$.
$$
0 \longrightarrow \mathcal{O}_Y \longrightarrow \mathcal{E}_{X, \Delta, f} \longrightarrow \mathcal{T}_{X, \Delta, f} \longrightarrow 0.
$$
The reflexive sheaf $\mathcal{E}_{X, \Delta, f}$ is the desired sheaf called the \emph{adapted canonical extension}. We have the following properties for this sheaf.

\begin{lemm}\label{orbietextensionproperties}
    With notation as above:
    
     $(i)$ If $f$ is orbi-\'etale, we have $\mathcal{E}_{X, \Delta, f} = \mathcal{E}_Y$ (see \cite[Sect.~4]{GKP22} to see how $\mathcal{E}_Y$ is construct when $Y$ is klt).

     $(ii)$ If $\gamma \colon Y' \longrightarrow Y$ is a quasi-\'etale morphism with $Y'$ normal, we have $\mathcal{E}_{X, \Delta, \gamma \circ f} = \gamma^{[*]} \mathcal{E}_{X, \Delta, f}$ (we define $\gamma^{[*]} \mathcal{E}_{X, \Delta, f} := \left( \gamma^* \mathcal{E}_{X, \Delta, f} \right)^{\vee \vee}$).
\end{lemm}

\begin{proof}
$(i)$
When $f$ is orbi-\'etale, we have $\Omega^{[1]}_Y \simeq \Omega^{[1]}_{X, \Delta, f}$ by \cite[Lemma.~18.2]{CGG24}. Moreover, we have $K_Y = f^* (K_X + \Delta)$ and the cocycle $\left( f^* \frac{\mathrm{d}h_{\alpha\beta}}{h_{\alpha\beta}} \right)$ defined previously correspond to those (up to some exact cocycles) defined in \cite[Sect.~4]{GKP22}, as they are cocycles associated to the divisor $-m K_Y$. Hence, they produce the same class in $H^1(Y, \Omega^{[1]}_{X, \Delta, f})$. It follows that the extensions $\mathcal{E}_{X, \Delta, f}$ and $\mathcal{E}_Y$ are isomorphic.

$(ii)$
We observe that the pullback of the previous cocycles by $\gamma$ are given by $\left(  \gamma^* f^* \frac{\mathrm{d} h_{\alpha\beta}}{h_{\alpha\beta}} \right)$. They provide cocycles that define $\mathcal{E}_{X, \Delta, \gamma \circ f}$. 

Let $U \subseteq Y'_\reg$ be an open subset such that $\gamma$ is \'etale on $U$ and $(f \circ \gamma)^{-1}(\Delta + H) $ (here $H$ denote the additional branching locus of $f$) is a smooth hypersurface on $U$. We can choose $U$ in a way that its complement has codimension at least 2. The explicit description of $\Omega^{[1]}_{X, \Delta, f}$ at general points (see \cite[Rem.~3.8]{CKT16}) and the \'etale property imply that we have an isomorphism $\gamma^* \Omega^{[1]}_{X, \Delta, f} \overset{\sim}{\longrightarrow} \Omega^{[1]}_{X, \Delta, f \circ \gamma}$, which maps the pullback of $\left( f^* \frac{\mathrm{d} h_{\alpha\beta}}{h_{\alpha\beta}} \right)$ on $\left( \gamma^*f^* \frac{\mathrm{d} h_{\alpha\beta}}{h_{\alpha\beta}} \right)$. Hence, both sheaves $\gamma^*\mathcal{E}_{X, \Delta, f \, | U}$ and $\mathcal{E}_{X, \Delta, \gamma \circ f \, |U}$ coincide. It follows that we have the isomorphism $\gamma^{[*]}\mathcal{E}_{X, \Delta, f} \simeq \mathcal{E}_{X, \Delta, \gamma \circ f}$ because both sheaves are reflexive.
\end{proof}

\subsubsection{Canonical extension of a klt pair}

We defined in Example~\ref{deforbifoldextension} a canonical extension when $(X, \Delta)$ is an orbifold, but there is a way to construct $\mathcal{E}_{(X, \Delta)}$ if we assume that it is merely a klt pair with standard coefficients.

We explained in Remark~\ref{firstorbistructurekltpair} that the pair $(X, \Delta)$ can be endowed with an orbi-\'etale orbistructure, we denote it by $\lbrace (U_\alpha, V_\alpha, \eta_\alpha) \rbrace_\alpha$ and use notation of Definition~\ref{deforbistructure}. We set $\mathcal{E}_\alpha := \mathcal{E}_{V_\alpha}$. Because the morphisms $g_{\alpha\beta}$ are quasi-\'etale and as a consequence of \cite[Rem.~4.6]{GKP22}, we have the following isomorphisms:
$$
g_{\alpha\beta}^{[*]} \mathcal{E}_\alpha \simeq \mathcal{E}_{V_{\alpha\beta}} \simeq g_{\beta\alpha}^{[*]}\mathcal{E}_\beta,
$$
that produce an orbisheaf on $(X, \Delta)$.

\newpage

\subsection{Pullback of an orbisheaf by a strictly adapted morphism}

\subsubsection{First construction of the adapted sheaf}

Let $(X, \Delta)$ be a klt pair and let $f \colon Y \longrightarrow X$ be a strictly $\Delta$-adapted morphism. With Remark~\ref{firstorbistructurekltpair}, we can endow $(X, \Delta)$ with an orbi-\'etale orbistructure, and refined it to have a smooth one above $(X, \Delta)_\orb$.

\begin{defi}\label{defadaptedsheaf}
	Let $\mathcal{F}$ be an orbibundle on $(X, \Delta)_\orb$. As $f$ is an orbifold morphism between $(f^{-1}((X, \Delta)_\orb, 0)$ and $(X, \Delta)_\orb$, we can consider the pullback orbibundle $f^* \mathcal{F}$. Let $i \colon f^{-1}((X, \Delta)_\orb) \hookrightarrow Y$ be the canonical inclusion. \emph{The adapted sheaf associated to $\mathcal{F}$}, denoted by $\mathcal{F}_{X, \Delta, f}$ is a sheaf on $Y$ defined by:
    $
    \mathcal{F}_{X, \Delta, f} := i_* ((f^* \mathcal{F})^\inv).
    $
\end{defi}

The pullback orbistructure is orbi-\'etale on $(f^{-1}((X, \Delta)_\orb), 0)$, so that the following isomorphism of orbisheaves holds: $\mathcal{F}_{X, \Delta, f}^{[\orb]} \simeq f^* \mathcal{F}$ as a consequence of Proposition~\ref{factsinvariantsheaves}.

\subsubsection{A refinement of the pullback orbistructure}

In order to describe the pullback of orbibudles, we refine the pullback orbistructure above a big open subset $Y$ in order to get a smooth one, which is well-suited to compute orbifold Chern classes with Chern--Weil theory.\\

\noindent
\emph{Setup.}\label{setup} Let $(X, \Delta)$ be a klt pair, and let $f \colon Y \longrightarrow X$ be a strictly $\Delta$-adapted morphism with $(Y, 0)$ klt. Let us endow $(X, \Delta)_\orb$ with a smooth orbi-\'etale orbistructure. We denote by $\lbrace (S_\alpha, R_\alpha, \nu_\alpha) \rbrace_\alpha$ the pullback orbistructure on $(f^{-1}((X, \Delta)_\orb), 0)$.
Let us justify that the space $R_\alpha$ is klt. Indeed, the morphism $\nu_\alpha \colon R_\alpha \longrightarrow Y$ is finite and satisfies $K_{R_\alpha} = \nu_\alpha^* K_{Y}$ by quasi-\'etaleness. We get that $(R_\alpha, 0)$ is klt as a consequence of \cite[Prop.~5.20]{KM08}.

By \cite[Prop.~9.3]{GKKP11}, there is an open subset $R_{\alpha \; \orb}$ which has only quotient singularities, and its complement has codimension at least 3. Its image by $\nu_\alpha$ is open in $Y$ by open mapping theorem. Hence, the set $Y^\circ = \bigcup_\alpha \nu_\alpha(R_{\alpha \, \orb} )$ is open, and its complement is included in an analytic set of codimension at least 3. By taking smooth orbi-\'etale orbistructures on each $(R_{\alpha \, \orb}, 0)$ and taking Galois closures (see \cite[Cor.~27]{CGG24}) of the composition of a l.u.c. by $\nu_\alpha$, we get that there are smooth spaces $W_{\alpha\alpha'}$ above $R_{\alpha \, \orb}$ such that $\mu_{\alpha\alpha'} \colon W_{\alpha\alpha'} \longrightarrow Y^\circ$ is a Galois quasi-\'etale morphism.

This construction gives an orbistructure on $(Y^\circ, 0)$ in which $f$ becomes an orbifold morphism between $Y^\circ$ and $(X,\Delta)_\orb$ in the sense of Definition~\ref{orbifoldmorphism}. To simplify notation, we will denote by $\mu_{\alpha'} \colon W_{\alpha'} \longrightarrow \mu_{\alpha'}(W_{\alpha'}) =: T_{\alpha'} \subseteq Y^\circ$ a morphism of this refined orbistructure, $W_{\alpha'\beta'}$ the normalization of the fiber product and $l_{\alpha'\beta'} \colon W_{\alpha'\beta'} \longrightarrow W_{\alpha'}$ the corresponding \'etale morphisms.

$$
\xymatrix{
    W_{\alpha\alpha'} \text{ smooth } \ar[d] \ar@/_3pc/[dd]_{\mu_{\alpha\alpha'}} \ar@{.>}[drrr]^{f_{\alpha\alpha'}} &&& \\
    R_{\alpha \; \orb} \ar[d]_{\nu_\alpha} \ar@{^{(}->}[r] & R_\alpha \ar[d]_{\nu_\alpha} \ar[rr]_{f_\alpha} && V_\alpha \text{ smooth } \ar[d]^{\eta_\alpha} \\
    Y^\circ \ar@{^{(}->}[r] & f^{-1}((X, \Delta)_{\orb}) \ar[rr]_f && (X, \Delta)_{\orb}
    }
$$ 

\begin{rema}
	Let $\mathcal{F}$ be an orbibundle defined on $(X, \Delta)_\orb$. The adapted sheaf $\mathcal{F}_{X, \Delta, f}$ can be constructed in the following way: we can pullback $\mathcal{F}$ by $f$ on this refined orbistructure, and taking the sheaf of invariants. It provides a sheaf on $Y^\circ$ that coincides with $\mathcal{F}_{X, \Delta, f}$ on $Y^\circ$. This is a convenient orbistructure, because we deal with genuine orbibundles above orbifolds.
\end{rema}

\subsubsection{Comparison with classical adapted sheaves}

In this section we show the following lemma:

\begin{prop}\label{equivalenceadaptedsheavespullbacks}
    With notation as above, we have the isomorphisms of orbisheaves above $Y^\circ$:
    $$
    f^* \Omega^1_{(X, \Delta)_\orb} \simeq \left( \Omega^{[1]}_{X, \Delta, f} \right)^{[\orb]} 
    \;\;\;\;
    \text{ and }
    \;\;\;\;
    f^* \mathcal{E}_{(X, \Delta)_\orb} \simeq \left( \mathcal{E}_{X, \Delta, f} \right)^{[\orb]}.
    $$    
\end{prop}

\begin{proof}
    \noindent
    \underline{$f^* \Omega^1_{(X, \Delta)_\orb} \simeq \left( \Omega^{[1]}_{X, \Delta, f} \right)^{[\orb]} $ :}
    It suffices to consruct an isomorphism of orbisheaves above an open subset of $Y^\circ$ whose complement is of codimension at least 2. Indeed, it will yield an isomorphism of sheaves, which can be extended by reflexivity of both sheaves. Then, it suffices to construct this isomorphism on l.u.c. above $Y_\reg$ and general points of ramification divisor.
    
    The strategy to construct such isomorphisms is to identify them as subsheaves of $\Omega^1_{W_{\alpha'}}$, and verify that these isomorphisms are compatible with transition maps.

    First, we describe such a local isomorphism for general points above $H$. We use notation of \S~\ref{setup} and those of Definition~\ref{orbifoldmorphism} and Remark~\ref{remarkoforbifoldmorphism} for the orbifold morphism $f$. We can write in normal forms:
    $$
    \xymatrix{
    \lbrace (z_1,w_1, \underline{z}) \;\;|\;\; z_1^N = w_1 \rbrace \subseteq W_{\alpha'} \ar@{|->}[dd]_{\mu_{\alpha'}}^\sim \ar@{|->}[rr]^{f_{\alpha\alpha'}}&& w \in V_\alpha \ar@{=}[dd]^{\eta_\alpha} \\
    \\
     z = (z_1, \underline{z}) \in Y \ar@{|->}[rr]_{f} &&  (z_1^N, \underline{z}) = w \in H
    }
    $$
    We have two inclusions $f_{\alpha\alpha'}^* \Omega^1_{V_\alpha} \hookrightarrow \Omega^1_{W_{\alpha'}}$ and $\mu_{\alpha'}^* \Omega^{[1]}_{X, \Delta, f} \hookrightarrow \Omega^1_{W_{\alpha'}}$ These inclusions have the same image: it is a locally free sheaf generated by $(z_1^{N-1} \mathrm{d}z_1, \mathrm{d}z_2, \ldots, \mathrm{d}z_n)$.
    
    Secondly, above a general point of $\Delta$, the morphisms $f_{\alpha\alpha'}$ and $\mu_{\alpha'}$ are \'etale, and we have naturally an isomorphism $f_{\alpha\alpha'}^* \Omega^1_{V_\alpha} \simeq \Omega^1_{W_{\alpha'}} \simeq \mu_{\alpha'}^* \Omega^{[1]}_{X, \Delta, f}$.\\
    Now, on more complicated l.u.c. $V_\alpha$ and $W_{\alpha'}$, the previous cases corresponds to normal forms of the morphisms in codimension one. Hence, we have an isomorphism on a big open subset of $W_{\alpha'}$, that extends as $\mu_{\alpha'}^{[*]} \Omega^{[1]}_{X, \Delta, f}$ and $f_{\alpha\alpha'}^* \Omega^1_{V_\alpha}$ are reflexive on $W_{\alpha'}$. Hence, we have a collection of isomorphisms $\theta_{\alpha'} \colon \mu^{[*]}_{\alpha'} \Omega^{[1]}_{X, \Delta, f} \overset{\sim}{\longrightarrow} f^*_{\alpha\alpha'} \Omega^1_{V_\alpha}$.\footnote{The morphism $\theta_{\alpha'}$ is totally determined by the index $\alpha'$, even it seems to depend on an extra index $\alpha$. Indeed, if $A$ is the set of indices $\alpha$ parametrizing the orbistructure on $(X, \Delta)_\orb$, the set of indices of the orbistructure on $(Y^\circ, 0)$ is of the form $A' := \bigsqcup_{\alpha \in A} B_\alpha$, where $B_\alpha$ is the set of indices of the smooth orbistructure on $R_{\alpha \, \orb}$. It follows that there is a map:
    $$
    \iota \colon \alpha' \in A' \longmapsto  \left( \alpha\textrm{ such that } \alpha' \in B_\alpha \right) \in A. $$ The pullback $f_{\alpha\alpha'} \Omega^{1}_{V_\alpha}$ is in fact equals to $f_{\iota(\alpha')\alpha'} \Omega^{1}_{V_\iota(\alpha')}$.}
    Now, let us check the compatibility on intersections $W_{\alpha'\beta'}$. As previously, we can show the compatibility by identifying subsheaves of $\Omega^1_{W_{\alpha'\beta'}}$ on big open subsets of $W_{\alpha'\beta'}$. Let us recall that we denote by $l_{\alpha'\beta'} \colon W_{\alpha'\beta'} \longrightarrow W_{\alpha'}$ the morphisms of the orbistructure on $Y^\circ$ and $f_{\alpha\beta, \alpha'\beta'} \colon W_{\alpha'\beta'} \longrightarrow V_{\alpha\beta}$ defined in Remark~\ref{remarkoforbifoldmorphism}. We have the following diagram of sheaves on $W_{\alpha'\beta'}$: 
    $$
    \xymatrix{
    l_{\alpha'\beta'}^*f_{\alpha\alpha'}^* \Omega^1_{V_{\alpha}} \ar@{^{(}->}[rd] \ar[rr] \ar@/^3pc/[rrrr]^{f_{\alpha\beta,\alpha'\beta'}^*p_{\beta\alpha}}_{\sim} && f_{\alpha\beta,\alpha'\beta'}^* \Omega^1_{V_{\alpha\beta}} \ar[rr] \ar@{^{(}->}[d] && l_{\beta'\alpha'}^* f_{\beta\beta'}^* \Omega^1_{V_\beta} \ar@{^{(}->}[ld]  \\
     & l_{\alpha'\beta'}^* \Omega^1_{W_{\alpha'}} \ar@{=}[r] & \Omega^1_{W_{\alpha'\beta'}} \ar@{=}[r] & l_{\beta'\alpha'}^* \Omega^1_{W_{\beta'}} & \\
    l_{\alpha'\beta'}^*\mu_{\alpha'}^{[*]} \Omega^{[1]}_{X, \Delta, f} \ar[uu]_{l_{\alpha'\beta'}^* \theta_{\alpha'}} \ar@{^{(}->}[ur] \ar[rrrr]_\sim^{\mathrm{nat}_{\beta'\alpha'}} &&&& l_{\beta'\alpha'}^*\mu_{\beta'}^{[*]} \Omega^{[1]}_{X, \Delta, f} \ar[uu]_{l_{\beta'\alpha'}^* \theta_{\beta'}} \ar@{^{(}->}[lu]
    }
    $$
    By compatibility of the pullback of forms with composition, the morphisms in the previous diagram commute, and we get the compatibility condition.\\

    \noindent
    \underline{$f^* \mathcal{E}_{(X, \Delta)_\orb} \simeq \left( \mathcal{E}_{X, \Delta, f} \right)^{[\orb]}$ :}   
    We want to show that there exists a collection of isomorphisms:
    $$
    \Theta_{\alpha'} \colon \mathcal{O}_{W_{\alpha'}} \oplus f_{\alpha\alpha'}^* \Omega^1_{V_\alpha} \overset{\sim}{\longrightarrow} \mathcal{O}_{T_{\alpha'}} \oplus \mu_{\alpha'}^* \Omega^{[1]}_{X, \Delta, f} 
    $$
    defined on the refined orbisructure above $Y^\circ$, that are compatible with transition maps.
    We can define $\Theta_{\alpha'} :=
    \left[
    \begin{array}{c|ccc}
        1 && 0 & \\
        \hline
        &&& \\
        0 && \theta_{\alpha'} & \\ 
        &&& 
    \end{array}
    \right]$.

    It remains to show that these $\Theta_{\alpha'}$ are compatible with transition maps. Both orbisheaves are given by collections of 1-forms $(\omega_{\alpha'\beta'})$ and $(\xi_{\alpha'\beta'})$ on each $W_{\alpha'\beta'}$. Up to an exact cocycle, they are equal as they can be obtained as pullbacks on $W_{\alpha'\beta'}$ of the same collection $\left( \frac{\mathrm{d} h_{\alpha\beta}}{h_{\alpha\beta}} \right)$, where $(h_{\alpha\beta})$ are the transition maps defining the line bundle $-m(K_X + \Delta)$ on $(X, \Delta)$ for $m \gg 1$.
\end{proof}

\begin{rema}
    We use notation of \S~\ref{setup}. On $Y^\circ$, if $\mathcal{F} = \Omega^1_{(X, \Delta)_\orb}$ (resp. $ \mathcal{F} = \mathcal{E}_{(X, \Delta)_\orb}$), then we have the isomorphism for the adapted sheaf $\mathcal{F}_{X, \Delta, f} = \Omega^{[1]}_{X, \Delta, f}$ (resp. $\mathcal{F}_{X, \Delta, f} = \mathcal{E}_{X, \Delta, f}$) as an application of the previous lemma, and using functoriality of $(.)^\textrm{inv}$ detailed in Definition~\ref{sheafofinvariants} and the point A of Proposition~\ref{factsinvariantsheaves}.
    As these classical adapted sheaves are algebraic coherent on $Y$, we infer by \cite[Thm.~2]{Ser66} that $\mathcal{F}_{X, \Delta, f}$ are coherent, reflexive sheaves on $Y$, and isomorphic to $\Omega^{[1]}_{X, \Delta, f}$ (resp. $\mathcal{E}_{X, \Delta, f}$) on the whole $Y$.   
\end{rema}

\begin{rema}\label{remarkpullbackcomposition}
    Let us consider a finite morphism $\gamma: Y' \longrightarrow Y$ such that the composition $f \circ \gamma$ is strictly $\Delta$-adapted. We use notation of Definition~\ref{pullbackstructure}.
    As the orbistructures $\left( \mathcal{C}_f \right)_\gamma$ and $\mathcal{C}_{f \circ \gamma}$ coincide, we have the natural isomorphism on this orbistructure $\gamma^* f^* \mathcal{F} = (f \circ \gamma)^* \mathcal{F}$. 

    Now, let us justify the isomorphism: $\mathcal{F}_{X, \Delta, f \circ \gamma} \simeq \gamma^{[*]} \mathcal{F}_{X, \Delta, f}$.
    One can observe that $\gamma^* (\mathcal{F}_{X, \Delta, f}^\orb) \simeq (\gamma^* \mathcal{F}_{X, \Delta, f})^{\orb}$ by pulling back sheaves on l.u.c. above a big open subset $Y'^\circ$ in two different ways. Because $\gamma$ and morphisms of the local uniformizing charts are quasi-flat, these local sheaves are moreover reflexive away from an analytic subset of codimension at least 2. It follows that the reflexive hulls $(\gamma^{[*]} \mathcal{F}_{X, \Delta, f})^{[\orb]}$ and $\gamma^{[*]} (\mathcal{F}_{X, \Delta, f}^{[\orb]})$ are isomorphic in codimension 1, hence everywhere. Moreover, these isomorphisms are compatible with transition morphisms on a big open subset, hence everywhere by reflexivity. To sum up, we get an isomorphism $\gamma^{[*]}\left( \mathcal{F}_{X, \Delta, f}^{[\orb]} \right) \simeq \left( \gamma^{[*]} \mathcal{F}_{X, \Delta, f} \right)^{[\orb]} $.
    It implies by taking invariant sections and using the fact that: $\gamma^{[*]}\left( \mathcal{F}_{X, \Delta, f}^{[\orb]} \right) \simeq \gamma^* f^* \mathcal{F}$ that $\mathcal{F}_{X, \Delta, f \circ \gamma} \simeq \gamma^{[*]} \mathcal{F}_{X, \Delta, f}$.
\end{rema}

Now, we can show that assumption $(i)$ in the main theorem implies that the semistability property transmits to any strictly $\Delta$-adapted morphism.

\begin{prop}\label{propagationsemistabilityforanyadap}
    Let $(X, \Delta)$ be a log Fano pair such that the assumption $(i)$ in the main theorem is satisfied. Then, for any strictly $\Delta$-adapted morphism $f \colon Y \longrightarrow X$, the adapted canonical extension $\mathcal{E}_{X, \Delta, f}$ is semistable with respect to $f^*(-(K_X + \Delta))$.
\end{prop}

\begin{proof}
    Let $f_0 \colon Y_0 \longrightarrow X$ be the strictly $\Delta$-adapted morphism of the assumption $(i)$ of our main theorem, and $Z$ the normalization of $Y_0 \underset{X}{\times} Y$. Let $F \colon Z \longrightarrow Y$, $F_0 \colon Z \longrightarrow Y_0$ be the natural projections so that $f \circ F = f_0 \circ F_0$.
    First, we have an isomorphism of sheaves on $Z$: $F^{[*]} \mathcal{E}_{X, \Delta, f} \simeq F_0^{[*]} \mathcal{E}_{X, \Delta, f_0}$. Indeed, one can observe with Remark~\ref{remarkpullbackcomposition} and Proposition~\ref{equivalenceadaptedsheavespullbacks}:
    $$
    F^{[*]} \mathcal{E}_{X, \Delta, f} \simeq \mathcal{E}_{X, \Delta, f \circ F} \simeq \mathcal{E}_{X, \Delta, f_0 \circ F_0} \simeq F_0^{[*]} \mathcal{E}_{X, \Delta, f_0}.
    $$

    Now, we argue by contradiction, following the proof of \cite[Prop.~3.6]{GT22}. Assume that $\mathcal{E}_{X, \Delta, f}$ is not $f^* (-(K_X + \Delta))$-semistable. Hence, using arguments of \cite[Fact~5.1.3]{GKP22}, the sheaf $F^{[*]} \mathcal{E}_{X, \Delta, f}$ is not $F^* f^* (-(K_X + \Delta))$-semistable. Hence, the previous isomorphism implies that $F_0^{[*]} \mathcal{E}_{X, \Delta, f_0}$ is not $F_0^* f_0^* (-(K_X + \Delta))$-semistable. The polarization is invariant under the action of $\mathrm{Gal}(F_0)$, hence the maximal destabilizing subsheaf, denoted $\mathcal{F}$, is $\mathrm{Gal}(F_0)$-invariant. By \cite[Thm.~4.2.15]{HL10}, it follows that there exists a subsheaf $\mathcal{G}$ of $\mathcal{E}_{X, \Delta, f_0}$ such that $\mathcal{F} = f_0^* \mathcal{G}$ that destabilizes $\mathcal{E}_{X, \Delta, f_0}$, contradicting its semistability.
\end{proof}

\subsubsection{Computation of orbifold Chern classes of $\mathcal{F}_{X, \Delta, f}$}

We can compute klt Chern classes of $\mathcal{F}_{X, \Delta, f}$. Let us first recall what are Chern classes for klt pairs.

\begin{defi}
    Let $(X, \Delta)$ be a klt pair, and $\mathcal{F}$ an orbibundle on an open subset $U \subseteq (X, \Delta)_\orb$, whose complement is included in an analytic subset of codimension at least 3.\\
    Then, for $i = 1,2$, we have an isomorphism $\iota \colon H^{2n-2i}_{\textrm{dR}, \orb, c}(U, \C) \simeq H^{2n-2i}(X, \C)$ (see \cite[Section.~5]{GK20}) that provides a linear form $\widehat{c_i}(\mathcal{F}) \in H^{2n-2i}(X, \C)^\vee$ defined by:
    $$
    \widehat{c_i}(\mathcal{F}) \cdot \kappa := \int_U \Omega \wedge A,
    $$
    where $\kappa \in H^{2n-2i}(X, \C)$, $A$ is an orbifold $(2n-2i)$-form with compact support in $U$ which represents the class $\iota^{-1}(\kappa)$, and $\Omega$ is an orbifold $(2i)$-form representing the orbifold Chern class $c_i^\orb(\mathcal{F})$.
    
    The intersection $c_1^\orb (\mathcal{F})^2$ yields similarly a class $\widehat{c_1}^2(\mathcal{F}) \in H^{2n-4}(X, \C)^\vee$.
    We denote $c_i(X, \Delta) = \widehat{c_i}(\mathcal{T}_{(X, \Delta)_\orb})$.
\end{defi}

As $K_X + \Delta$ is $\Q$-Cartier, we can compute the class:
$$ c_1(-(K_X + \Delta)) := \frac{1}{m}c_1(-m(K_X + \Delta)) \in H^2(X, \C),$$ for $m \gg 1$. It is independent of $m$. This class yields by intersection pairing an element of $H^{2n-2}(X, \C)^\vee$ which coincides with $c_1(X, \Delta)$. This definition allows us to consider $c_1(X, \Delta)^2 \in H^{2n-4}(X, \C)^\vee$, as the intersection with $c_1(-(K_X + \Delta))^2$.

\begin{prop}\label{calculchernextension}
    With notation as above, we have:
    $$
    \widehat{c_1}(\mathcal{F}_{X, \Delta, f}) \cdot f^* \kappa = (\deg f) \widehat{c_1}(\mathcal{F}) \cdot \kappa,
    $$
    $$
    \widehat{c_1}^2(\mathcal{F}_{X, \Delta, f}) \cdot f^* \kappa = (\deg f) \widehat{c_1}^2(\mathcal{F}) \cdot \kappa,
    $$
    $$
    \widehat{c_2}(\mathcal{F}_{X, \Delta, f}) \cdot f^* \kappa = (\deg f) \widehat{c_2}(\mathcal{F}) \cdot \kappa,
    $$
\end{prop}

\begin{proof}
    We show the equality for $\widehat{c_2}$, the other ones are obtained similarly.\\
    Let $h$ be a Hermitian metric on $\mathcal{F}$. There is a pullback orbifold hermitian metric $f^*h = (f_\alpha^* h_\alpha)$ on $f^*\mathcal{F}$. Let $A$ be a $(2n-4)$-orbifold form which represents $\kappa$ with compact support in $f(Y^\circ)$. The previous constructions show that we can integrate $\widehat{c_1}(\mathcal{F}_{X, \Delta, f}) \cdot f^* \kappa$ on $Y^\circ$. By removing subsets of measure zero, we get:
    \begin{align*}
     \widehat{c_2}(\mathcal{F}_{X, \Delta, f}) \cdot f^* \kappa 
     & = \int_{Y^{**}} c_2^{\orb}(\mathcal{F}_{X, \Delta, f} , f^* h)_{| Y^{**}} \wedge f^* A \\
     & = (\deg f) \int_{X^{**}}  c_2^{\orb}(\mathcal{F} , h)_{| X^{**}} \wedge A = (\deg f)\; \widehat{c_2}(\mathcal{F}) \cdot \kappa.
    \end{align*}
    Here $X^{**} := X_\reg \backslash \textrm{Br}(f) $, where $\textrm{Br}(f)$ is the branching locus of $f$, and $Y^{**} = f^{-1}(X^{**})$. We used the fact the morphism $f_{| Y^{**}}$ is \'etale, finite of degree $\deg f$.
\end{proof}

A consequence of this computation is:
$$
\widehat{c_1}(\mathcal{E}_{X, \Delta, f}) \cdot f^* \kappa = \widehat{c_1}(\mathcal{T}_{X, \Delta, f}) \cdot f^* \kappa = (\deg f) c_1(X, \Delta) \cdot \kappa,
$$
$$
\widehat{c_1}^2(\mathcal{E}_{X, \Delta, f}) \cdot f^* \kappa = \widehat{c_1}^2(\mathcal{T}_{X, \Delta, f}) \cdot f^* \kappa = (\deg f) c_1^2(X, \Delta) \cdot \kappa,
$$
$$
\widehat{c_2}(\mathcal{E}_{X, \Delta, f}) \cdot f^* \kappa = \widehat{c_2}(\mathcal{T}_{X, \Delta, f}) \cdot f^* \kappa = (\deg f) c_2(X, \Delta) \cdot \kappa.
$$

\section{Proof of the main Theorem}

In this section we prove the main theorem. Let $(X, \Delta)$ be a log Fano pair and we denote the following assertions:\\
    $(\mathrm{Quo})$: \emph{There exists a finite subgroup $G$ of $\PGL(n+1, \C)$ such that $(X, \Delta) \simeq (\PP^n / G, \Delta_G)$},\\
    $(\mathrm{Unif})$: \emph{The pair $(X, \Delta)$ satisfies assumptions $(i)$ and $(ii)$}.

\subsection{\texorpdfstring{$(\mathrm{Quo})$}{(Quo)} implies \texorpdfstring{$(\mathrm{Unif})$}{(Unif)}}

Let $G$ be a finite subgroup of $\PGL(n+1, \C)$ such that $(X, \Delta)$ is isomorphic to $(\PP^n / G, \Delta_G)$.
We get that the quotient map $ \pi_G \colon \PP^n \longrightarrow \mathcal{P}_G$ is orbi-\'etale, and we have:
$$
\pi_G^* \Omega^1_{\mathcal{P}_G} = \Omega^1_{\PP^n}
\;\;\;\;
\text{ and }
\;\;\;\;
\pi_G^* \mathcal{E}_{\mathcal{P}_G} = \mathcal{E}_{\PP^n},
$$
as a consequence of Proposition~\ref{equivalenceadaptedsheavespullbacks}, and an application of \cite[Lemma.~18]{CGG24} and Lemma~\ref{orbietextensionproperties}. In particular, we have $\mathcal{E}_{\PP^n/G, \Delta_G, \pi_G} = \mathcal{E}_{\PP^n}$ which is semistable with respect to $\pi_G^* c_1(\mathcal{P}_G) = c_1(\PP^n)$ (see \cite[Thm.~0.1]{Tia92}, with the fact that the Fubini--Study metric $\omega_{FS}$ is K\"ahler--Einstein). Hence, we get that assumption $(i)$ is satisfied.

Moreover, we have:
\begin{align*}
  ( 2(n+1) & c_2^\orb(X, \Delta) - n c_1^\orb(X, \Delta)^2 )  \cdot c_1^\orb(X, \Delta)^{n-2} \\
  & = \frac{1}{|G|}  \left( 2(n+1)c_2^\orb(\PP^n) - n c_1^\orb(\PP^n)^2 \right) \cdot c_1^\orb(\PP^n)^{n-2} = 0.
\end{align*}

Hence, the pair satisfies Miyaoka--Yau equality $(ii)$.

\begin{rema}
Let us show that the group $G$ can be recovered from $(X, \Delta)$, as $G = \pi_1^\orb(X, \Delta)$. 
In the following we denote by $\textrm{Fix}\, G = \left\lbrace p \in \PP^n \,\,|\,\, \exists g \in G \backslash \lbrace 1 \rbrace, g \cdot x = x \right\rbrace$.
If $\Delta_G = 0$, it is a consequence of \cite[Lemma.~7.8]{GKP22}.

If $\Delta_G \neq 0$, the morphism $\pi_G$ is \'etale on $\PP^n \backslash \textrm{Fix}\, G$. It follows that we have a short exact sequence:
$$
1 \longrightarrow \pi_1(\PP^n \backslash \textrm{Fix}\, G) \overset{{\pi_G}_*}{\longrightarrow} \pi_1(\mathcal{P}_G^* ) \longrightarrow G \longrightarrow 1
$$
Let us describe the injection. The space $\PP^n \backslash \textrm{Fix}\, G$ is topologically homeomorphic to $\C^n \backslash \bigcup V_i $ where $V_i$ are affine subspaces associated to $\Delta_i$. This space is generated by meridian loops $\delta_i$ around $V_i$'s of codimension 1. Now, the image of the loop $\delta_i$ around $V_i$ is $\gamma_i^{m_i}$ by ${\pi_G}_*$, as a consequence of local forms of the projection around general points of divisors. Hence, we have: $G \simeq \pi_1(\mathcal{P}_G^*) / \langle\langle \gamma_i^{m_i} \rangle\rangle = \pi_1^\orb(\mathcal{P}_G)$.
\end{rema}

\subsection{\texorpdfstring{$(\mathrm{Unif})$}{(Unif)} implies \texorpdfstring{$(\mathrm{Quo})$}{(Quo)}}
Let $(X, \Delta)$ be a log Fano pair satisfying assumptions $(i)$ and $(ii)$ of the main theorem.
The proof is divided in two steps.\\

\noindent
\textbf{Step 1: $(X, \Delta)$ is an orbifold.}\\
Proposition~\ref{elimdelta} states that there exists a very ample line bundle $L$ which produces many orbi-\'etale morphisms above an open set of $X$. Let $x \in X$, we pick a divisor $H \in |L|$ such that $x \notin H$ (it is possible since $L$ is basepoint-free). A consequence of this proposition is the existence of a strictly $\Delta$-adapted morphism $f \colon Y \longrightarrow X$ that is orbi-\'etale for the pair $\displaystyle \left( X, \Delta + \left( 1 - \frac{1}{N} \right) H \right)$ for some integer $N \gg 1$. We can assume that $Y$ is maximally quasi-\'etale, since we can compose by a maximally quasi-\'etale cover \cite[Facts 5.1.1-3]{GKP22} and take the Galois closure of the morphism.\cite[Cor.~27]{CGG24}\footnote{We can replace $Y \longrightarrow X$ by the composition $Y' \longrightarrow Y \longrightarrow X$ where $Y' \longrightarrow Y$ is the maximally quasi-\'etale covering of $Y$. It follows that the composition is branched exactly along $\Delta + \left( 1 - \frac{1}{N} \right) H$. We can take the Galois closure $Y'' \longrightarrow Y' \longrightarrow X$, and this map is orbi-\'etale for this divisor, hence strictly $\Delta$-adapted.}. We show that $Y \backslash f^{-1}(H)$ is smooth.

By Proposition~\ref{calculchernextension}, and assumption $(ii)$ in the main theorem, we have:
\begin{align*}
& \left( 2(n+1) \widehat{c_2}(\mathcal{E}_{X, \Delta, f}) - n\widehat{c_1}(\mathcal{E}_{X, \Delta, f})^2 \right) \cdot \left[ f^*(-(K_X + \Delta)) \right]^{n-2} \\
& = (\deg f) (2(n+1)c_2(X, \Delta) - nc_1(X, \Delta)^2) \cdot c_1(X, \Delta)^{n-2} \\
& = 0.
\end{align*}

The morphism $f$ being strictly $\Delta$-adapted, the extension $\mathcal{E}_{X, \Delta, f}$ is semistable with respect to $f^*(-(K_X + \Delta))$ as a consequence of Proposition~\ref{propagationsemistabilityforanyadap}. Its rank is $n+1$ and it satisfies the Bogomolov--Gieseker equality. By \cite[Prop. 1.6]{GKP22}, the extension ${\mathcal{E}_{X, \Delta, f}}_{|Y_{\reg}}$ is locally free and projectively flat. So, $\End(\mathcal{E}_{X, \Delta, f})$ is locally free and flat \cite[Cor. 3.6]{GKP22}. Now, the extension $\mathcal{E}_{X, \Delta, f}$ splits locally, we have on sufficiently small open subsets $U \subseteq Y \backslash f^{-1}(H)$:
$$
\End(\mathcal{E}_{X, \Delta, f})_{| U} = \mathcal{O}_U \oplus \mathcal{T}_{X, \Delta, f |U} \oplus \Omega^{[1]}_{X, \Delta, f |U} \oplus \End( \mathcal{T}_{X, \Delta, f})_{| U}
$$
and $\mathcal{T}_{X, \Delta, f}$ is locally free as a direct summand of a locally free sheaf. On $Y \backslash f^{-1}(H)$, there is an isomorphism: $\mathcal{T}_{X, \Delta, f} \simeq \mathcal{T}_Y$. By \cite[Thm 1.1]{Dru14}, the space $Y \backslash f^{-1}(H)$ is smooth. Moreover, the morphism $f \colon Y \backslash f^{-1}(H) \longrightarrow X \backslash H$ is orbi-\'etale, providing a suitable l.u.c. around $x$ to make a smooth orbi-\'etale orbistructure on $(X, \Delta)$.\\

\noindent
\textbf{Step 2: $(X, \Delta) \simeq (\PP^n/G, \Delta_G)$ }\\
By \cite[Thm. 2]{Bra21}, the log Fano pair $(X, \Delta)$ has a finite orbifold fundamental group. So, the universal covering map $\pi \colon (\widetilde{X}_\Delta, \widetilde{\Delta}) \longrightarrow (X, \Delta)$ is finite.\\

\noindent
\underline{$\widetilde{\Delta} = 0$ and $\widetilde{X}_\Delta$ is smooth:}\\
Let us denote $\alpha := -(K_X + \Delta)$ and we consider the orbibundle $\mathcal{F} := \End(\mathcal{E}_{(X, \Delta)})$. Let $f \colon Y \longrightarrow X$ be a morphism constructed in the previous step (so that its restriction to $Y \backslash f^{-1}(H)$ provides us with a morphism of the smooth orbi-\'etale orbistructure). In particular, we assume that $Y$ is maximally quasi-\'etale.

One can observe that the corresponding reflexive sheaf $\mathcal{F}_{X, \Delta, f}$ constructed in Definition~\ref{defadaptedsheaf} is $\End(\mathcal{E}_{X, \Delta, f})$. Indeed, we observe first that pullback of orbibundles commute with $\End$. Let us justify that $(.)^\inv$ commute with $\End$. Using notation of \ref{sheafofinvariants}, if $\varphi_{\alpha\beta}$ are transition maps of an orbibundle, it suffices to observe that the construction of local invariant sheaves commute with duals when they are locally free, and the transition morphisms of the dual of invariants are $\theta_{\beta\alpha}^\top$ that coincide with the restriction of $\varphi_{\beta\alpha}^\top$ to invariant sections. It follows that we have the isomorphism:
$$
\left( f^* \End(\mathcal{E}_{(X, \Delta)}) \right)^\inv = \End( f^*\mathcal{E}_{(X, \Delta)})^\inv = \End \left( \left( f^*\mathcal{E}_{(X, \Delta)} \right)^\inv \right).
$$
Finally, we have: $\mathcal{F}_{X, \Delta, f} \simeq \End \left( \mathcal{E}_{X, \Delta, f} \right)$, and this isomorphism holds on $Y$ as both sheaves are reflexive.

Now, we can observe as a consequence of \cite[Prop.~4.4]{GKP16} that this sheaf is semistable with respect to $f^* \alpha$. Moreover, it satisfies:
$$
\widehat{\ch_1}(\mathcal{F}_{X, \Delta, f}) \cdot f^* \alpha^{n-1} = \widehat{\ch_2}(\mathcal{F}_{X, \Delta, f}) \cdot f^* \alpha^{n-2} = 0.
$$
We deduce from the non-abelian Hodge correspondence for klt projective spaces that $\mathcal{F}_{X, \Delta, f}$ can be endowed with a flat holomorphic connection $\nabla$ on $Y_\reg$, that is $\mathrm{Gal}(f)$-equivariant by functoriality of the correspondence (see Theorems 3.4 and Corollary 3.8 of \cite{GKPTa19}). Hence, with notation of \S~\ref{setup}, it yields an orbiconnection $(\nu_\alpha^* \nabla)$ on the restriction of the pullback orbistructure to $Y_\reg$ (namely, the connections $\nu_\alpha^* \nabla$ are defined above ${S_\alpha}_\reg$ on the bundles $\nu_\alpha^* \mathcal{F}_{X, \Delta, f} \simeq f_\alpha^* \mathcal{F}_\alpha$). 

By descent theorems for connections \cite[Prop.~2.7]{GKPTb19}, the orbibundle $\mathcal{F}_\alpha$ is endowed with a flat connection $\nabla_\alpha$ on an open subset of $V_\alpha$ whose complement has codimension at least 2. The collection $(\nabla_\alpha)$ satifies moreover compatibility conditions on $V_{\alpha\beta}$, as it suffices to consider the conditions on $R_{\alpha\beta}$ between $f_\alpha^* \nabla_\alpha$ and $f_\beta^* \nabla_\beta$, and to take the functor of $G$-invariant sections, in order to descend them on a big open subset of $V_{\alpha\beta}$.
Because removing an analytic subset of codimension 2 does not change the fundamental group and the flat bundles are in correspondence with representations of the fundamental group, we obtain that the connection $\nabla_\alpha$ extends through the whole $V_\alpha$, and the compatibility conditions too. Finally, the orbisheaf $\mathcal{F}$ is endowed with a flat orbiconnection. Hence, by Corollary~\ref{Endcanflat}, we get that $\widetilde{\Delta} = 0$ and $\widetilde{X}_\Delta$ is smooth.\\

\noindent
\underline{$\widetilde{X}_\Delta \simeq \PP^n$:}\\
Finally, the endomorphism bundle $\End(\mathcal{E}_{\widetilde{X}_\Delta}) = \pi^* \End(\mathcal{E}_{(X, \Delta)})$ is flat (in particular, it is trivial). It follows by \cite[Prop.~3.7]{GKP22} that the bundle $\mathcal{E}_{\widetilde{X}_\Delta}$ is projectively flat. Because $\widetilde{X}_\Delta$ is simply connected, we obtained that there exists a line bundle $\mathcal{L}$ such that $\mathcal{E}_{\widetilde{X}_\Delta} = \mathcal{L}^{\oplus n+1}$. Hence, we have $c_1(\widetilde{X}_\Delta) = c_1(\mathcal{E}_{\widetilde{X}_\Delta}) = (n+1)c_1(\mathcal{L})$, and we conclude that $\widetilde{X}_\Delta \simeq \PP^n$ by \cite[Corollary of Thm 1.1]{KO73}.\\

Finally, we have an isomorphism $(X, \Delta) \simeq (\PP^n/G, \Delta_G)$ where $G = \pi_1^{\orb}(X, \Delta)$ is a finite group. \qed

\section{Examples}

In this section, we introduce examples of weighted projective spaces, and we characterize log smooth surfaces which are quotient of $\PP^2$.

\subsection{Weighted projective spaces}

Here, we study the examples of weighted projective spaces. Let $\textbf{a} = (a_0, \ldots, a_n) \in \Z_{\geq 1}^{n+1}$ be $n+1$ positive integers. There is an action of $\C^\times = \C \backslash \lbrace 0 \rbrace$ on $\C^{n+1} \backslash \lbrace 0 \rbrace$ given by:
$$
\lambda \cdot (x_i)_{0 \leq i \leq n} = (\lambda^{a_i}x_i)_{0 \leq i \leq n}
$$
The weighted projective space $\PP(\textbf{a})$ is the orbit space of this action. We denote by $\left[ x_i \right]_\textbf{a}$ the class of $(x_i)$ in $\PP(\textbf{a})$.

There is an order on $\Z_{\geq 1}^{n+1}$ defined by: $\textbf{a} | \textbf{b} $ iff $\forall 0 \leq i \leq n, a_i | b_i $.
If $\textbf{a} | \textbf{b}$, there is a map:
$$
\pi_{\textbf{a} | \textbf{b}} \colon \left[ x_i \right]_\textbf{a} \in \PP(\textbf{a}) \longmapsto \left[ x_i^{\frac{b_i}{a_i}} \right]_\textbf{b} \in \PP(\textbf{b}),
$$
which is the quotient map of the action of $\displaystyle \prod_{0 \leq i \leq n} \mu_{b_i/a_i}$ where $\mu_m$ is the group of $m$-th root of unity in $\C$. The action is given by the diagonal action:
$$
(\zeta_i)_{0 \leq i \leq n} \cdot \left[ x_i \right]_\textbf{a} := \left[ \zeta_i x_i \right]_\textbf{a}.
$$
The kernel of this action is $\displaystyle \bigcap_{0 \leq i \leq n} \mu_{\frac{b_i}{a_i}} = \mu_{\gcd (\frac{b_i}{a_i})}$. Hence, the action is effective if and only if $\gcd \left( \frac{b_i}{a_i} \right) = 1$.

If $\textbf{a} = (1, \ldots, 1)$, we have $\textbf{a} | \textbf{b}$ for any $\textbf{b}$, and $\PP(\textbf{a}) = \PP^n$. It follows that the spaces $\PP(\textbf{b})$ are quotient of $\PP^n$ by the diagonal action of roots of the unity. Hence, each $\PP(\textbf{b})$ is obtained as a finite quotient of $\PP^n$.
    
\begin{exem}\label{examplewps}
    If $\textbf{b} = (m,\ldots,m)$, we have $\PP(\textbf{b}) = \PP^n$, and the map $\pi_{\textbf{a}|\textbf{b}}$ is the quotient by the diagonal action of $(\mu_m)^{n+1}$. This action is not effective, and it can be reduced to the action of $G := (\mu_m)^n$ on $\PP^n$ given by:
    $$
    (\zeta_1, \ldots, \zeta_n) \cdot \left[ x_i \right]_{0 \leq i \leq n} = \left[ x_0 : \zeta_1 x_1 : \ldots : \zeta_n x_n \right],
    $$
    and the quotient map is given by $\left[ x_i \right] \in \PP^n \longmapsto \left[ x_i^m \right] \in \PP^n$.

    The set of fixed points is:
    $$
    \textrm{Fix}\, G := \bigcup_{g \in G \backslash \lbrace \textrm{id} \rbrace} \left\lbrace \left[ x \right] \in \PP^n | g \cdot \left[ x \right] = \left[ x \right] \right\rbrace = \bigcup_{0 \leq k \leq n} (x_k = 0).
    $$
    Hence, it follows that $\displaystyle \mathcal{P}_G = \left( \PP^n ,\sum_{0 \leq k \leq n} \left( 1 - \frac{1}{m} \right) (x_k = 0) \right)$
\end{exem}

\begin{exem}
    If $\textbf{b} = (1,1,2)$, the group $G := \mu_2$ acts on $\PP^2$ as following:
    $$
    (-1) \cdot \left[ x_0 : x_1 : x_2 \right] = \left[ x_0 : x_1 : -x_2 \right]
    $$
    The set of fixed points under this action is $\textrm{Fix}\, (-1) = (x_2 = 0) \cup \lbrace \left[ 0 : 0 : 1 \right] \rbrace$. One can show that the following morphism:
    $$
    \left[ x_0 : x_1 : x_2 \right] \in \PP^2 \longmapsto \left[ x_0^2 : x_0 x_1 : x_1^2 : x_2^2 \right] =: \left[ u : v : w : t \right] \in \PP^3 
    $$
    factors through $G$. 
    
    Moreover, it induces an isomorphism between $\PP^2 / G$ and the cone $(v^2 = uw)$ in $\PP^3$. It is a singular variety with a conical singularity in $\left[ 0 : 0 : 0 : 1 \right]$.
    
    The branching divisor $\Delta_G$ is the image of $(x_2 = 0)$ by the morphism above, it corresponds to $\Delta_1 := (v^2 = uw) \cap (t=0)$, which is a smooth conic in the projective plane $(t = 0)$. Finally, we get $\displaystyle \mathcal{P}_G = \left( (v^2 = uw) , \frac{1}{2} \Delta_1 \right)$.
\end{exem}

\subsection{Log smooth surfaces which are quotients of \texorpdfstring{$\PP^2$}{P2}}

In this section we establish the following lemma.

\begin{lemm}
    Let $(X, \Delta)$ be a log smooth surface which is a quotient of the form $\mathcal{P}_G$.
    Then, we have $X = \PP^2$, and there exists an integer $m$ such that $\Delta = \left( 1 - \frac{1}{m} \right) (H_0 + H_1 + H_2)$ where $H_i$ are hyperplanes in general position.
    In particular, $G = \mu_m \times \mu_m $, where the action is that described in Example~\ref{examplewps}, conjugated by an automorphism of $\PP^2$.
\end{lemm}

\begin{proof}
    First, by \cite[Example.~1.4]{BHPVdV04}, we have $X = \PP^2$. Let us denote by $\Delta = \sum_{1 \leq i \leq k} \left( 1 - \frac{1}{m_i} \right) \Delta_i$ where $\Delta_i$ is an irreducible curve of degree $d_i$, and $m_i \geq 2$. For simplicity, we set $\lambda_i := \left( 1 - \frac{1}{m_i} \right)$.
    
    We compute the discriminant of Miyaoka--Yau $\Delta_{MY} := 3 c_2(X, \Delta) - c_1(X, \Delta)^2$. Because $(X, \Delta)$ is a quotient of $\PP^2$, we have $\Delta_{MY} = 0$. We will show that this happens only for $\Delta$ as in the statement.
    
    The pair $(X, \Delta)$ is log smooth, so by \cite[Example~3.9]{GT22}, and using that $c_1(\PP^2) = -3 \left[ \omega_{FS} \right]$ and $c_2(\Omega^1_{\PP^2}) = 3 \left[ \omega_{FS} \right]^2$:
    
    \begin{align*}
        c_1(X, \Delta)^2 & = \left(-3 + \sum_{1 \leq i \leq k} \lambda_i d_i \right)^2 \\
        c_2(X, \Delta) & = 3  - 3 \sum_i \lambda_i d_i + \sum_i \lambda_i d_i^2 + \sum_{1 \leq i < j \leq k} \lambda_i\lambda_j d_i d_j \\
    \end{align*}

    We get:
    \begin{align*}
        \Delta_{MY} & = -3 \sum_i \lambda_i d_i + 3 \sum_i \lambda_i d_i^2 - \sum_{1 \leq i \leq k} \lambda_i^2 d_i^2 + \sum_{1 \leq i < j \leq k} \lambda_i\lambda_j d_i d_j \\
        & = \sum_i \lambda_i d_i \left( -3 + 3 d_i - \lambda_i d_i \right) + \sum_{1 \leq i < j \leq k} \lambda_i\lambda_j d_i d_j
    \end{align*}

    First, we can show that at least one of the $d_i$'s equals one. If not, we have the following inequality: $-3 + 3 d_i - \lambda_i d_i \geq -3 + 3 d_i - d_i = 2d_i - 3 > 0 $, and it follows that $\Delta_{MY} > 0$ that contradicts that $\Delta_{MY} = 0$.

    In fact, each $d_i$ equals one. If not, let us assume that $d_1 = \cdots = d_t = 1 < d_{t+1} \leq \cdots \leq d_k$. We have by the lower bound $d_i \geq 1$: 
    \begin{align*}
        \Delta_{MY} & \geq \sum_{1 \leq i \leq t} - \lambda_i^2 + \sum_{1 \leq i \leq k-1} \lambda_i \lambda_k d_k \\
        & + \sum_{t+1 \leq i \leq k} \lambda_i d_i \left( -3 + 3 d_i - \lambda_i d_i \right) + \sum_{1 \leq i < j \leq k-1} \lambda_i\lambda_j d_i d_j
    \end{align*}
    As $\lambda_i^2 < \lambda_i \leq \lambda_i \lambda_k d_k$, the sum of the first two sums is positive, and it follows from the previous remarks that $\Delta_{MY} > 0$.

    Now we study the discriminant according to the number $k$ of hyperplanes in $\Delta$.
    In the following, we assume up to change indices that $\lambda_1 \leq \cdots \leq \lambda_k$.
    
    If $k=1$, we have $\Delta_{MY} = - \lambda_1^2 < 0$.

    If $k = 2$, we have $\Delta_{MY} = - \lambda_1^2 - \lambda_2^2 + \lambda_1 \lambda_2 = - \left( \lambda_1 - \lambda_2 \right)^2 - \lambda_1 \lambda_2 < 0 $.

    If $k \geq 4$, one can observe that we have for $1 \leq i \leq k-1$ the following inequality: $\lambda_i^2 \leq \lambda_i\lambda_{i+1}$. Hence, we have:
    \begin{align*}
        \Delta_{MY} & > - \sum_{1 \leq i \leq k} \lambda_i^2 + \sum_{1 \leq i \leq k-1} \lambda_i \lambda_{i+1} + \lambda_k\lambda_1 + \lambda_k\lambda_2 \\
        & \geq -\lambda_k^2 + \lambda_k\lambda_1 + \lambda_k\lambda_2 = \lambda_k (\lambda_1 + \lambda_2 - \lambda_k)
    \end{align*}

    Because $\lambda_i \in \left[ \frac{1}{2} , 1 \right)$, we have $\lambda_1 + \lambda_2 - \lambda_k > 0$. Hence, $\Delta_{MY} > 0$.

    If $k = 3$, we have:
    \begin{align*}
    \Delta_{MY} & = - \lambda_1^2 - \lambda_2^2 - \lambda_3^2 + \lambda_1\lambda_2 + \lambda_1\lambda_3 + \lambda_2\lambda_3 \\
    & = - \frac{1}{2} \left( (\lambda_1 - \lambda_2)^2 + (\lambda_1 - \lambda_3)^2 + (\lambda_2 - \lambda_3)^2 \right)
    \end{align*}
    We have $\Delta_{MY} = 0$ iff $\lambda_1 = \lambda_2 = \lambda_3$, which is equivalent to $m_1 = m_2 = m_3$. Hence, $\Delta = \left( 1 - \frac{1}{m} \right) (H_0 + H_1 + H_2)$. As $\Delta$ is a simple normal crossing divisor, the three hyperplanes $H_i$ are in general position. Hence, there exists $A \in \PGL(3, \C)$ such that $A (z_i = 0) = H_i$. Hence, the quotient map induced by the action of $G = (\mu_m )^2$ described in the previous section conjugated by $A$ branches exactly along $H_i$ with multiplicity $m_i$. Hence, it follows that $\mathcal{P}_G$ is the quotient of $\PP^2$ by this action conjugated by $A$.
\end{proof}

\printbibliography

@Book{KM08,
 Author = {Koll{\'a}r, J{\'a}nos and Mori, Shigefumi},
 Title = {Birational geometry of algebraic varieties. {With} the collaboration of {C}. {H}. {Clemens} and {A}. {Corti}},
 Edition = {Paperback reprint of the hardback edition 1998},
 FSeries = {Cambridge Tracts in Mathematics},
 Series = {Camb. Tracts Math.},
 ISSN = {0950-6284},
 Volume = {134},
 ISBN = {978-0-521-06022-6},
 Year = {2008},
 Publisher = {Cambridge: Cambridge University Press},
 Language = {English},
 zbMATH = {5273473},
 Zbl = {1143.14014}
}

@Article{GKP22,
 Author = {Greb, Daniel and Kebekus, Stefan and Peternell, Thomas},
 Title = {Projective flatness over klt spaces and uniformisation of varieties with nef anti-canonical divisor},
 FJournal = {Journal of Algebraic Geometry},
 Journal = {J. Algebr. Geom.},
 ISSN = {1056-3911},
 Volume = {31},
 Number = {3},
 Pages = {467--496},
 Year = {2022},
 Language = {English},
 DOI = {10.1090/jag/785},
 zbMATH = {7528018},
 Zbl = {1490.14013}
}

@Article{GKP16,
 Author = {Greb, Daniel and Kebekus, Stefan and Peternell, Thomas},
 Title = {{\'E}tale fundamental groups of {Kawamata} log terminal spaces, flat sheaves, and quotients of abelian varieties},
 FJournal = {Duke Mathematical Journal},
 Journal = {Duke Math. J.},
 ISSN = {0012-7094},
 Volume = {165},
 Number = {10},
 Pages = {1965--2004},
 Year = {2016},
 Language = {English},
 DOI = {10.1215/00127094-3450859},
 zbMATH = {6617538},
 Zbl = {1360.14094}
}

@Article{Bla96,
 Author = {Blache, Raimund},
 Title = {Chern classes and {Hirzebruch}-{Riemann}-{Roch} theorem for coherent sheaves on complex-projective orbifolds with isolated singularities},
 FJournal = {Mathematische Zeitschrift},
 Journal = {Math. Z.},
 ISSN = {0025-5874},
 Volume = {222},
 Number = {1},
 Pages = {7--57},
 Year = {1996},
 Language = {English},
 DOI = {10.1007/BF02621857},
 zbMATH = {922495},
 Zbl = {0949.14006}
}

@Article{Sat56,
 Author = {Satake, Ichir\^o},
 Title = {On a generalization of the notion of manifold},
 FJournal = {Proceedings of the National Academy of Sciences of the United States of America},
 Journal = {Proc. Natl. Acad. Sci. USA},
 ISSN = {0027-8424},
 Volume = {42},
 Pages = {359--363},
 Year = {1956},
 Language = {English},
 DOI = {10.1073/pnas.42.6.359},
 zbMATH = {3123695},
 Zbl = {0074.18103}
}

@Article{GT22,
 Author = {Guenancia, Henri and Taji, Behrouz},
 Title = {Orbifold stability and {Miyaoka}-{Yau} inequality for minimal pairs},
 FJournal = {Geometry \& Topology},
 Journal = {Geom. Topol.},
 ISSN = {1465-3060},
 Volume = {26},
 Number = {4},
 Pages = {1435--1482},
 Year = {2022},
 Language = {English},
 DOI = {10.2140/gt.2022.26.1435},
 zbMATH = {7629605},
 Zbl = {1505.14035}
}

@Article{Bra21,
 Author = {Braun, Lukas},
 Title = {The local fundamental group of a {Kawamata} log terminal singularity is finite},
 FJournal = {Inventiones Mathematicae},
 Journal = {Invent. Math.},
 ISSN = {0020-9910},
 Volume = {226},
 Number = {3},
 Pages = {845--896},
 Year = {2021},
 Language = {English},
 DOI = {10.1007/s00222-021-01062-0},
 zbMATH = {7433732},
 Zbl = {1479.14029}
}

@Article{Dru14,
 Author = {Druel, St{\'e}phane},
 Title = {The {Zariski}-{Lipman} conjecture for log canonical spaces},
 FJournal = {Bulletin of the London Mathematical Society},
 Journal = {Bull. Lond. Math. Soc.},
 ISSN = {0024-6093},
 Volume = {46},
 Number = {4},
 Pages = {827--835},
 Year = {2014},
 Language = {English},
 DOI = {10.1112/blms/bdu040},
 zbMATH = {6324455},
 Zbl = {1357.14009}
}

@Article{KO73,
 Author = {Kobayashi, Shoshichi and Ochiai, Takushiro},
 Title = {Characterizations of complex projective spaces and hyperquadrics},
 FJournal = {Journal of Mathematics of Kyoto University},
 Journal = {J. Math. Kyoto Univ.},
 ISSN = {0023-608X},
 Volume = {13},
 Pages = {31--47},
 Year = {1973},
 Language = {English},
 DOI = {10.1215/kjm/1250523432},
 zbMATH = {3411441},
 Zbl = {0261.32013}
}

@Book{Kob87,
 Author = {Kobayashi, Shoshichi},
 Title = {Differential geometry of complex vector bundles},
 FSeries = {Publications of the Mathematical Society of Japan},
 Series = {Publ. Math. Soc. Japan},
 ISSN = {0549-4540},
 Volume = {15},
 ISBN = {0-691-08467-X},
 Year = {1987},
 Publisher = {Princeton, NJ: Princeton University Press; Tokyo: Iwanami Shoten Publishers},
 Language = {English},
 zbMATH = {44936},
 Zbl = {0708.53002}
}

@misc{CKT16,
      title={Generic positivity and applications to hyperbolicity of moduli spaces}, 
      author={Benoît Claudon and Stefan Kebekus and Behrouz Taji},
      year={2016},
      eprint={1610.09832},
      archivePrefix={arXiv},
      primaryClass={math.AG}
}

@Article{SY22,
 Author = {Shen, Shu and Yu, Jianqing},
 Title = {Flat vector bundles and analytic torsion on orbifolds},
 FJournal = {Communications in Analysis and Geometry},
 Journal = {Commun. Anal. Geom.},
 ISSN = {1019-8385},
 Volume = {30},
 Number = {3},
 Pages = {575--656},
 Year = {2022},
 Language = {English},
 DOI = {10.4310/CAG.2022.v30.n3.a3},
 zbMATH = {7634452},
 Zbl = {1512.58016}
}

@Article{ES18,
 Author = {Eyssidieux, Philippe and Sala, Francesco},
 Title = {Instantons and framed sheaves on {K{\"a}hler} {Deligne}-{Mumford} stacks},
 FJournal = {Annales de la Facult{\'e} des Sciences de Toulouse. Math{\'e}matiques. S{\'e}rie VI},
 Journal = {Ann. Fac. Sci. Toulouse, Math. (6)},
 ISSN = {0240-2963},
 Volume = {27},
 Number = {3},
 Pages = {599--628},
 Year = {2018},
 Language = {English},
 DOI = {10.5802/afst.1579},
 zbMATH = {6979712},
 Zbl = {1467.14032}
}

@Article{MP97,
 Author = {Moerdijk, I. and Pronk, D. A.},
 Title = {Orbifolds, sheaves and groupoids},
 FJournal = {\(K\)-Theory},
 Journal = {\(K\)-Theory},
 ISSN = {0920-3036},
 Volume = {12},
 Number = {1},
 Pages = {3--21},
 Year = {1997},
 Language = {English},
 DOI = {10.1023/A:1007767628271},
 zbMATH = {1059650},
 Zbl = {0883.22005}
}

@Article{GK20,
 Author = {Graf, Patrick and Kirschner, Tim},
 Title = {Finite quotients of three-dimensional complex tori},
 FJournal = {Annales de l'Institut Fourier},
 Journal = {Ann. Inst. Fourier},
 ISSN = {0373-0956},
 Volume = {70},
 Number = {2},
 Pages = {881--914},
 Year = {2020},
 Language = {English},
 DOI = {10.5802/aif.3326},
 zbMATH = {7210773},
 Zbl = {1445.14028}
}

@Article{GKKP11,
 Author = {Greb, Daniel and Kebekus, Stefan and Kov{\'a}cs, S{\'a}ndor J. and Peternell, Thomas },
 Title = {Differential forms on log canonical spaces},
 FJournal = {Publications Math{\'e}matiques},
 Journal = {Publ. Math., Inst. Hautes {\'E}tud. Sci.},
 ISSN = {0073-8301},
 Volume = {114},
 Pages = {87--169},
 Year = {2011},
 Language = {English},
 DOI = {10.1007/s10240-011-0036-0},
 zbMATH = {6112185},
 Zbl = {1258.14021}
}

@Article{DGP24,
 Author = {Druel, St{\'e}phane and Guenancia, Henri and P{\u{a}}un, Mihai},
 Title = {A decomposition theorem for {{\(\mathbb{Q}\)}}-{Fano} {K{\"a}hler}-{Einstein} varieties},
 FJournal = {Comptes Rendus. Math{\'e}matique. Acad{\'e}mie des Sciences, Paris},
 Journal = {C. R., Math., Acad. Sci. Paris},
 ISSN = {1631-073X},
 Volume = {362},
 Number = {S1},
 Pages = {93--118},
 Year = {2024},
 Language = {English},
 DOI = {10.5802/crmath.612},
 zbMATH = {7863256}
}

@Book{BHPVdV04,
 Author = {Barth, Wolf P. and Hulek, Klaus and Peters, Chris A. M. and Van de Ven, Antonius},
 Title = {Compact complex surfaces},
 Edition = {2nd enlarged ed.},
 FSeries = {Ergebnisse der Mathematik und ihrer Grenzgebiete. 3. Folge},
 Series = {Ergeb. Math. Grenzgeb., 3. Folge},
 ISSN = {0071-1136},
 Volume = {4},
 ISBN = {3-540-00832-2},
 Year = {2004},
 Publisher = {Berlin: Springer},
 Language = {English},
 zbMATH = {2008523},
 Zbl = {1036.14016}
}

@Book{HL10,
 Author = {Huybrechts, Daniel and Lehn, Manfred},
 Title = {The geometry of moduli spaces of sheaves},
 Edition = {2nd ed.},
 ISBN = {978-0-521-13420-0},
 Year = {2010},
 Publisher = {Cambridge: Cambridge University Press},
 Language = {English},
 zbMATH = {5765056},
 Zbl = {1206.14027}
}

@Article{Mor79,
 Author = {Mori, Shigefumi},
 Title = {Projective manifolds with ample tangent bundles},
 FJournal = {Annals of Mathematics. Second Series},
 Journal = {Ann. Math. (2)},
 ISSN = {0003-486X},
 Volume = {110},
 Pages = {593--606},
 Year = {1979},
 Language = {English},
 DOI = {10.2307/1971241},
 zbMATH = {3657924},
 Zbl = {0423.14006}
}

@InCollection{SY18,
 Author = {Siu, Yum-Tong and Yau, Shing-Tung},
 Title = {Compact {K{\"a}hler} manifolds of positive bisectional curvature},
 BookTitle = {Complex geometry from Riemann to K\"ahler-Einstein and Calabi-Yau},
 ISBN = {978-1-57146-352-4},
 Pages = {349--364},
 Year = {2018},
 Publisher = {Somerville, MA: International Press; Beijing: Higher Education Press},
 Language = {English},
 zbMATH = {7034477},
 Zbl = {1415.32018}
}

@Article{Tia92,
 Author = {Tian, Gang},
 Title = {On stability of the tangent bundles of {Fano} varieties},
 FJournal = {International Journal of Mathematics},
 Journal = {Int. J. Math.},
 ISSN = {0129-167X},
 Volume = {3},
 Number = {3},
 Pages = {401--413},
 Year = {1992},
 Language = {English},
 DOI = {10.1142/S0129167X92000175},
 zbMATH = {90946},
 Zbl = {0779.53009}
}

@Article{Ser66,
 Author = {Serre, Jean-Pierre},
 Title = {Prolongement de faisceaux analytiques coh{\'e}rents},
 FJournal = {Annales de l'Institut Fourier},
 Journal = {Ann. Inst. Fourier},
 ISSN = {0373-0956},
 Volume = {16},
 Number = {1},
 Pages = {363--374},
 Year = {1966},
 Language = {French},
 DOI = {10.5802/aif.234},
 zbMATH = {3232142},
 Zbl = {0144.08003}
}

@Article{GKPTa19,
 Author = {Greb, Daniel and Kebekus, Stefan and Peternell, Thomas and Taji, Behrouz},
 Title = {Nonabelian {Hodge} theory for klt spaces and descent theorems for vector bundles},
 FJournal = {Compositio Mathematica},
 Journal = {Compos. Math.},
 ISSN = {0010-437X},
 Volume = {155},
 Number = {2},
 Pages = {289--323},
 Year = {2019},
 Language = {English},
 DOI = {10.1112/S0010437X18007923},
 zbMATH = {7159507},
 Zbl = {1443.14009}
}

@Article{GKPTb19,
 Author = {Greb, Daniel and Kebekus, Stefan and Peternell, Thomas and Taji, Behrouz},
 Title = {The {Miyaoka}-{Yau} inequality and uniformisation of canonical models},
 FJournal = {Annales Scientifiques de l'{\'E}cole Normale Sup{\'e}rieure. Quatri{\`e}me S{\'e}rie},
 Journal = {Ann. Sci. {\'E}c. Norm. Sup{\'e}r. (4)},
 ISSN = {0012-9593},
 Volume = {52},
 Number = {6},
 Pages = {1487--1535},
 Year = {2019},
 Language = {English},
 DOI = {10.24033/asens.2414},
 zbMATH = {7201733},
 Zbl = {1452.32032}
}

@misc{GP24,
      title={Uniformization of klt pairs by bounded symmetric domains}, 
      author={Patrick Graf and Aryaman Patel},
      year={2024},
      eprint={2410.12753},
      archivePrefix={arXiv},
      primaryClass={math.AG},
      url={https://arxiv.org/abs/2410.12753}, 
}

@Article{CGG24,
 Author = {Claudon, Beno\^it and Graf, Patrick and Guenancia, Henri},
 Title = {Equality in the {Miyaoka}-{Yau} inequality and uniformization of non-positively curved klt pairs},
 FJournal = {Comptes Rendus. Math{\'e}matique. Acad{\'e}mie des Sciences, Paris},
 Journal = {C. R., Math., Acad. Sci. Paris},
 ISSN = {1631-073X},
 Volume = {362},
 Number = {S1},
 Pages = {55--81},
 Year = {2024},
 Language = {English},
 DOI = {10.5802/crmath.599},
 zbMATH = {7863254},
 Zbl = {1544.32036}
}

\end{document}